\documentclass[11pt,reqno]{amsart}

\usepackage{bbm}
\usepackage{mathrsfs}
\usepackage{stmaryrd}
\usepackage{amscd,latexsym,amsthm,amsfonts,amssymb,amsmath,amsxtra}

\usepackage[hyperfootnotes=false, colorlinks, citecolor=RoyalBlue, linkcolor=Bittersweet, urlcolor=link]{hyperref}
\usepackage{color}
\usepackage[usenames,dvipsnames]{xcolor}

\usepackage{graphicx}
\usepackage[all]{xy}

\usepackage{setspace}

\usepackage{tikz}
\usetikzlibrary{calc, intersections}
\usepackage{tikz-cd}

\theoremstyle{plain}
\newtheorem{thm}{Theorem}[section]
\newtheorem{lem}[thm]{Lemma}
\newtheorem{cor}[thm]{Corollary}
\newtheorem{prop}[thm]{Proposition}

\theoremstyle{definition}
\newtheorem{defn}[thm]{Definition}

\newtheorem{conj}[thm]{Conjecture}

\theoremstyle{remark}
\newtheorem{remark}[thm]{Remark}

    \newcommand{\sfG}{{\mathsf{G}}}\newcommand{\sfH}{{\mathsf{H}}}
  
  \newcommand{\sfS}{{\mathsf{S}}}
\newcommand{\sfu}{{\mathsf{u}}}

    \newcommand{\BC}{{\mathbb {C}}} 
     \newcommand{\BF}{{\mathbb {F}}}
    \newcommand{\BG}{{\mathbb {G}}}

     \newcommand{\BR}{{\mathbb {R}}}

    \newcommand{\CA}{{\mathcal {A}}} \newcommand{\CB}{{\mathcal {B}}}

    \newcommand{\fg}{{\mathfrak{g}}} \newcommand{\fh}{{\mathfrak{h}}}

    \newcommand{\fm}{{\mathfrak{m}}} 
     \newcommand{\fp}{{\mathfrak{p}}}

     \newcommand{\fJ}{{\mathfrak{J}}}

     \newcommand{\fT}{{\mathfrak{T}}}

    \newcommand{\ad}{{\mathrm{ad}}}
    \newcommand{\Ad}{{\mathrm{Ad}}}

    \newcommand{\depth}{{\mathrm{depth}}}

    \newcommand{\der}{\mathrm{der}}

    \newcommand{\Gal}{{\mathrm{Gal}}} 
    \newcommand{\GL}{{\mathrm{GL}}}
    \newcommand{\Hom}{{\mathrm{Hom}}}
    
    \newcommand{\id}{{\mathrm{id}}}
     \newcommand{\ind}{{\mathrm{ind}}}
     \newcommand{\Int}{{\mathrm{Int}}}

    \newcommand{\Lie}{{\mathrm{Lie}}}

    \newcommand{\new}{{\mathrm{new}}} 
    
     \newcommand{\N}{{\mathrm{N}}}
    \newcommand{\nm}{{\mathrm{nm}}}

     \newcommand{\rank}{{\mathrm{rank}}}

    \newcommand{\red}{\mathrm{red}} \newcommand{\R}{{\mathrm{R}}}
    \newcommand{\ram}{{\mathrm{ram}}}

    \newcommand{\Res}{{\mathrm{Res}}}

    \renewcommand{\ss}{{\mathrm{ss}}}

    \newcommand{\op}{\mathrm{op}}

    \newcommand{\unip}{{\mathrm{unip}}}

    \newcommand{\wh}{\widehat}

    \newcommand{\ov}{\overline}

    \newcommand{\sk}{\medskip}
    \newcommand{\lra}{\longrightarrow}
    \newcommand{\ra}{\rightarrow} 
    
    \newcommand{\bs}{\backslash}

    \newcommand{\s}{\sk\noindent}

    \newcommand{\abs}[1]{\lvert#1\rvert}

\title{Distinguished regular  supercuspidal representations}
\author{Chong Zhang}
\begin{document}
\date{}
\maketitle

\begin{abstract}
Based on recent work of Kaletha, we apply Hakim--Murnaghan's result
to study distinguished regular supercuspidal representations of
tamely ramified reductive $p$-adic groups. Assuming $p$ is
sufficiently large, we obtain a necessary and sufficient condition
for regular supercuspidal representations to be distinguished. We
also investigate the relation between the distinction problem and the Langlands
functoriality, and confirm a conjecture of Lapid for regular
depth-zero or epipelagic supercuspidal representations.
\end{abstract}

\tableofcontents

\section{Introduction}
\subsection*{Overview}
Supercuspidal representations are central in the representation theory of reductive groups over non-archimedean local fields, and the latter theory is crucial to the theory of automorphic representations. On the other hand, the properties of distinguished representations
have become a main theme in the study of automorphic
representations, especially after Jacquet and his collaborators'
work on various automorphic periods. The basic question is to
determine when the representations are distinguished, which we call the
distinction problem for short. For tame supercuspidal
representations, Hakim and Murnaghan \cite{hm08} developed a general
theory for this problem. Regular supercuspidal representations were recently introduced by Kaletha \cite{kal}. Our purpose is to apply Hakim--Murnaghan's work to determine distinguished regular supercuspidal representations, and study the relation between the distinction problem and the Langlands functoriality. Below we will give a brief summary of prior works related to this paper.

Our first goal is to give a natural criterion to detect the distinction of the regular supercuspidal representations, and our method relies heavily on the ways to construct supercuspidal representations. For a
tamely ramified reductive $p$-adic group $G$, Yu \cite{yu01}, inspired by Adler's
prior work \cite{adl98}, obtained  a remarkable way to construct
supercuspidal representations using generic cuspidal data. These
representations are called tame supercuspidal representations. For $G=\GL_n$, Howe's classical result \cite{how77} shows that these representations can be parameterized by much simpler data.
For general reductive groups, people are trying to find out a
more explicit parametrization, e.g. see
\cite{mur11}. In his recent
work \cite{kal}, Kaletha considered a large subclass of tame supercuspidal
representations which he called regular supercuspidal
representations. He showed that these representations can be
parameterized by simpler data $(S,\mu)$, called tame regular
elliptic pairs, than generic cuspidal data. This can be viewed as a generalization of Howe's parametrization mentioned above.

For tame supercuspidal
representations, Hakim and Murnaghan \cite{hm08} gave a criterion to detect the distinction with respect to symmetric spaces, and even obtained a
multiplicity formula in terms of the generic cuspidal data.
Using this theory together with Howe's construction, Hakim and his
collaborators successfully obtained much simpler criterions in terms of
Howe's data for several typical involutions of $\GL_n$, e.g. see \cite{hj12,hak13} and earlier work
\cite{hm02a,hm02b}. To apply Hakim--Murnaghan's result to concrete examples, it needs a further delicate analysis.

On the other hand, distinguished representations have a conjectural deep relation with the Langlands functoriality, which is the so-called relative Langlands program nowadays. We refer to Sakellaridis and Venkatesh's work \cite{sv} for a systematic exposition. To understand this relation, we have to assume the local Langlands correspondence. For regular supercuspidal representations, Kaletha \cite{kal} showed how to organize them into the $L$-packets attached to the regular supercuspidal $L$-parameters in the framework of rigid inner twists, which generalizes previous works \cite{ree08,dr09,dr10,ka14,ry14,kal15}. Our second goal is to apply our criterion for distinguished regular supercuspidal representations to the relative Langlands program.

\subsection*{Main results} Now we give a more detailed introduction
to our results. Let $F$ be a non-archimedean local field of residual characteristic $p$.
We suppose that $p$ is sufficiently large, and refer to Section \ref{subsubsec. assumption}
for the precise assumptions on $p$.

Let $G$ be a tamely ramified connected reductive group over $F$,
$\theta$ an involution of $G$ defined over $F$, and $H=G^\theta$ the
closed subgroup of fixed points of $\theta$. For an irreducible
admissible representation $\pi$ of $G(F)$, we say that it is {\em
$H$-distinguished} if the space $\Hom_{H(F)}(\pi,{\bf1})$ is non-zero where ${\bf1}$ stands for the trivial representation of $H(F)$.

The first part of this article is concerned with the properties of
distinguished regular supercuspidal representations in terms of the
inducing data tame regular elliptic pairs. We refer to
\cite[\S3]{kal} or Section \ref{subsubsec. regular sc repns} for the
basic definitions and facts on regular supercuspidal representations.

To state our results clearly, let us first consider the depth-zero
case, which is one of the cornerstones of the whole theory. In this case, all regular depth-zero supercuspidal representations of
$G(F)$ are constructed from the data $(S,\mu)$, where $S$ is a maximally
unramified elliptic maximal torus of $G$ and $\mu$ a regular
depth-zero character of $S(F)$. The construction is based on the
Deligne--Lusztig representation $\kappa_{(S,\mu)}$ of the parahoric
subgroup $G(F)_{x,0}$ of $G(F)$  determined by $S$. After extending
$\kappa_{(S,\mu)}$ to a representation $\tilde{\kappa}_{(S,\mu)}$ of
$S(F)G(F)_{x,0}$, we obtain the regular depth-zero supercuspidal
representation
$\pi_{(S,\mu)}=\ind_{S(F)G(F)_{x,0}}^{G(F)}\tilde{\kappa}_{(S,\mu)}$.
The isomorphism class of $\pi_{(S,\mu)}$ depends only on the
$G(F)$-conjugate class of $(S,\mu)$.

\begin{thm}\label{thm. intro thm depth-zero}
The regular depth-zero supercuspidal representation $\pi_{(\dot{S},\dot{\mu})}$ is $H$-distinguished if and
only if $(\dot{S},\dot{\mu})$ is $G(F)$-conjugate to a pair $(S,\mu)$ such that $S$ is
$\theta$-stable and
$$\mu|_{S^\theta(F)}=\varepsilon_{S}.$$
\end{thm}

Here, for a $\theta$-stable maximally unramified elliptic maximal
torus $S$, the character $\varepsilon_{S}$ is a quadratic character
of $S^\theta(F)$ (see Definition \ref{defn. character varepsilon}),
whose appearance arises from Lusztig's solution \cite{lus90} of the
distinction problem for symmetric spaces over finite fields. Moreover the character
$\varepsilon_S$ satisfies the property that
$\varepsilon_S|_{S^{\theta,\circ}(F)}=1$ where $S^{\theta,\circ}$ is
the identity component of $S^\theta$. Due to this property, Theorem
\ref{thm. intro thm depth-zero} implies the following relation
between the contragredient representation $\pi^\vee$ and the
$\theta$-twisted representation $\pi\circ\theta$ of $\pi$.

\begin{cor}[Corollary \ref{cor. cor of main theorem}]\label{cor. intro 1}
Suppose that $\pi$ is an $H$-distinguished regular depth-zero
supercuspidal representation of $G(F)$. Then we have
$\pi^\vee\simeq\pi\circ\theta$.
\end{cor}

For general regular supercuspidal representations of $G(F)$ of positive depth, they are constructed from the data {\em tame regular elliptic pairs} $(S,\mu)$, where $S$ is a tame elliptic
maximal torus of $G$ and $\mu$ a character of $S(F)$ satisfying
certain conditions (cf. \cite[Definition 3.7.5]{kal} for the precise definition). If $G=\GL_n$, these data coincide with Howe's notion of admissible characters in \cite{how77}. There is a process, called {\em Howe
factorization} of $(S,\mu)$, which is a generalization of the Howe factorization lemma for $\GL_n$ (\cite[Lemma 11 and Corollary]{how77}), to produce a cuspidal generic $G$-datum
\begin{equation}\label{equ. G-datum}
\Psi=\left(\vec{G}=(G^0,...,G^d),\pi_{(S,\mu_0)},\vec{\phi}=(\phi_0,...,\phi_d)\right),
\end{equation}
such that $\pi_{(S,\mu_0)}$ is a regular depth-zero supercuspidal
representation of $G^0(F)$. We refer to \cite[\S3.6 and \S3.7]{kal} or Section \ref{subsubsec. tame regular elliptic pairs} for more details. By Yu's construction, this $G$-datum
$\Psi$ gives rise to a regular supercuspidal representation
$\pi_{(S,\mu)}$ whose isomorphism class depends only on the
$G(F)$-conjugate class of $(S,\mu)$. The following is our main
theorem on the distinction problem, which is a generalization of Theorem
\ref{thm. intro thm depth-zero}.

\begin{thm}[Theorem \ref{thm. main theorem}]\label{thm. intro main theorem}
The regular supercuspidal representation
$\pi_{(\dot{S},\dot{\mu})}$ is $H$-distinguished if and only if
$(\dot{S},\dot{\mu})$ is $G(F)$-conjugate to a tame regular elliptic
pair $(S,\mu)$ such that $S$ is $\theta$-stable and
$$\mu|_{S^\theta(F)}=\varepsilon_{S}\cdot\eta_{S}.$$
\end{thm}

Here $\varepsilon_S$ is the quadratic character of $S^{\theta}(F)$ introduced before, but with respect to $(G^0,S)$, and $\eta_S$ is also a quadratic character of $S^\theta(F)$, whose
appearance arises from Hakim--Murnaghan's work when the
representation is of positive depth. The reason that we do not have
a consequence of Theorem \ref{thm. intro main theorem} as Corollary
\ref{cor. intro 1} is due to the character $\eta_S$. At this moment,
we could not show $\eta_S|_{S(F)^{1+\theta}}=1$, while this
relation holds for all the examples in the literature as far as we
know. In particular, $\eta_S$ is the trivial character for
epipelagic supercuspidal representations. Therefore the analog of
Corollary \ref{cor. intro 1} holds for epipelagic supercuspidal
representations.

Let us outline the proof of Theorem \ref{thm. intro main theorem}.
Roughly speaking, modulo $G(F)$-conjugation, we can apply
Hakim--Murnaghan's work to reduce the distinction problem of
$\pi_{(S,\mu)}$ to that of $\tilde{\kappa}_{(S,\mu_\circ)}$ which is
a representation of $S(F)G^0(F)_{x,0}$. To ensure that the
distinction problem makes sense for
$\tilde{\kappa}_{(S,\mu_\circ)}$, we have to show that
$S(F)G^0(F)_{x,0}$ is $\theta$-stable. This point is guaranteed by
the distinction problem over finite fields (Lemma \ref{lem. disctinction finite field })
and liftings of $\theta$-stable tori over finite fields to those over
$p$-adic fields (Lemma \ref{lem. stable torus}). The solution of the
distinction problem of $\tilde{\kappa}_{(S,\mu_\circ)}$ relies on
Proposition \ref{prop. multiplicity}. The proof of Proposition
\ref{prop. multiplicity} is a little complicated, due to the
subtleness of the construction of $\tilde{\kappa}_{(S,\mu_\circ)}$.

\begin{remark}\label{rem. hakim's work}
After a preliminary version of this article was completed, we notice
that Hakim's most recent work \cite{haka} and \cite{hakb} provide a
new approach to the construction of tame supercuspidal
representations and its application to the distinction problem. One of
the main features of \cite{haka} is that it eliminates Howe's
factorizations in Yu's construction. This new construction improves
the main results of \cite{hm08}, see \cite[Theorem 2.0.1]{hakb}. It
would be interesting to see whether our results can be simplified or
generalized when combined with these developments.
\end{remark}

The second part of this article is concerned with the properties of
distinguished regular supercuspidal representations in terms of the
Langlands parameters. 
The philosophy of the relative Langlands program
\cite{sv} is that if $\pi$ is an $H$-distinguished representation
then the $L$-parameter $\varphi$ of $\pi$ should have more
symmetries. Roughly speaking, $\varphi$ should factor through the $L$-group of the symmetric space.
In other words, $\pi$ should be a Langlands
functorial lift from a representation of some group other than
$G(F)$. In the context of symmetric spaces, Lapid proposed a conjecture
which is easier to state. We learnt of this conjecture from \cite[Conjecture
1.1]{gla}. We believe that Lapid's conjecture is closely related to the relative Langlands program, since it reflects certain information on the symmetry of the parameter
$\varphi$. However it is beyond our scope to discuss this relationship. For Galois symmetric spaces, Prasad \cite{pra} formulated
a more precise conjecture in terms of the refined $L$-parameters, see
{\em loc. cit.} for more details.

\begin{conj}[Lapid]\label{conj. Lapid}
Let $G$ be a connected reductive group over a $p$-adic field $F$,
$\theta$ an involution of $G$ defined over $F$, and $H=G^\theta$.
Let $\pi$ be an admissible irreducible representation of $G(F)$ and
$\Pi_\varphi(G)$ the conjectural $L$-packet containing $\pi$.
Suppose that $\pi$ is $H$-distinguished. Then we have
$$\left\{\tau\circ\theta\ |\ \tau\in\Pi_\varphi(G)\right\}
=\left\{\tau^\vee\ |\ \tau\in\Pi_\varphi(G)\right\}.$$ In other words,
$\Pi_\varphi(G)$ is invariant under
$\tau\mapsto\tau^\vee\circ\theta$.
\end{conj}

Now we return to the context of regular supercuspidal
representations and retain the assumptions as before. We further
assume that $F$ has characteristic zero and $G$ is quasi-split. Kaletha \cite[\S5]{kal} defined the
notion {\em regular supercuspidal $L$-parameters} $\varphi$ for $G$.
For each rigid inner twist $(G',\xi,z)$ of $G$, he also constructed
$L$-packets $\Pi_\varphi(G')$ that consists of certain regular
supercuspidal representations of $G'(F)$. In this paper, we consider not merely the
distinction problem for $G$, but also for all other rigid inner
twists $(G',\xi,z)$ of $G$ such that the fixed involution $\theta$
of $G$ can be ``transferred'' to an involution $\theta'$ of $G'$. For this purpose,
we introduce the notion {\em rigid inner twists} of
$(G,H,\theta)$ in Section \ref{subsubsec. rigid inner twists of sym space}, which are denoted by $(G',H',\theta')$ where
$H'=(G')^{\theta'}$.  For each rigid inner twist $(G',H',\theta')$,
we can think about $H'$-distinction for the representations in
$\Pi_\varphi(G')$. Motivated by Conjecture \ref{conj. Lapid}, we
would like to know what the sets
$$\Pi_\varphi^\theta(G'):=\left\{\pi\circ\theta'\ |\
\pi\in\Pi_\varphi(G')\right\}\ \ \textrm{and}\ \
\Pi_\varphi^\vee(G'):=\left\{\pi^\vee\ |\
\pi\in\Pi_\varphi(G')\right\}$$ are, especially in terms of the
$L$-parameters. The answer is not surprising and has been long
expected. Let $^LC$ be the Chevalley involution of the $L$-group
$^LG$ and $^L\theta$ the involution of $^LG$ dual to $\theta$. Then
$^L\theta\circ\varphi$ and $^LC\circ\varphi$ are also regular
supercuspidal parameters. In Propositions \ref{prop. twisted
L-packets} and \ref{prop. contragredient L-packet}, we show that
$$\Pi_\varphi^\theta(G')=\Pi_{^L\theta\circ\varphi}(G')\ \ \textrm{and}\ \
\Pi_\varphi^\vee(G')=\Pi_{^LC\circ\varphi}(G').$$ Therefore, if
$\pi^\vee$ is isomorphic to $\pi\circ\theta'$ for some
representation $\pi$ in $\Pi_\varphi(G')$, the parameters
$^LC\circ\varphi$ and $^L\theta\circ\varphi$ are $\wh{G}$-conjugate,
which implies $\Pi_\varphi^\theta(G')=\Pi_\varphi^\vee(G')$. In
particular, if $\pi$ is an $H'$-distinguished regular depth-zero
or epipelagic supercuspidal representation then Conjecture \ref{conj.
Lapid} holds (see Corollary \ref{cor. consequence}).

\begin{remark}\label{rem. unramified galois}
In a sequel \cite{zhab} to this paper, we show that $\eta_S$ is trivial if $\theta$ is an unramified Galois involution (cf. {\em loc. cit.} Proposition 4.1). Therefore Conjecture \ref{conj. Lapid} holds for distinguished regular supercuspidal representations with respect to unramified Galois involutions (cf. {\em loc. cit.} Corollary 4.2).
\end{remark}

\subsection*{Organization of this article} The assumptions on the
residual characteristic $p$, and necessary notation and convention
are given in the rest of this section. We recollect some background
materials in Section \ref{sec. preliminaries}, including Yu's
construction of tame supercuspidal representations, Hakim--Murnaghan's
result on distinguished tame supercuspidal representations, and
Kaletha's work on regular supercuspidal representations. Some
details of these contents that we need will appear in latter
sections or be referred to the references. Our main results on
the distinction problem are stated in Section \ref{subsubsec. main theorems}.
Before that, we introduce the two characters $\varepsilon_S$ and
$\eta_S$, and analyze their properties in Sections \ref{subsubsec.
character varepsilon} and \ref{subsubsec. character eta}. The notion
$(\theta,\varepsilon\eta)$-symmetric pairs is introduced in Section
\ref{subsubsec. twisted theta symmetric} where its basic properties
are also discussed. The proofs of Theorems \ref{thm. intro thm
depth-zero} and \ref{thm. intro main theorem} are given in Section
\ref{subsec. proofs}. Kaletha's construction of regular
supercuspidal $L$-packets $\Pi_\varphi$ is reviewed in Section
\ref{subsec. regular supercuspidal L-packets}. Then we study the
twisted $L$-packets $\Pi_\varphi^\theta$ and the contragredient
$L$-packets $\Pi_\varphi^\vee$ in Sections \ref{subsec. twisted
L-packet} and \ref{subsec. contragredient L-packets} respectively.

\subsection*{Assumptions}\label{subsubsec. assumption}
Throughout this article, $F$ is a non-archimedean local field of residual characteristic $p$. 
Moreover,  we will require that $F$ has characteristic zero in Section 4. 
The reason that we restrict to characteristic zero is that this assumption is needed in Kaletha's construction of $L$-packets. To
apply the theories mentioned in the introduction, we have to make certain
restrictions on $p$ in different stages. We require that
$p$ satisfies all of the following conditions:
\begin{enumerate}
\item $p$ is odd,
\item $p\nmid\abs{\pi_1(G_\der)}$,
\item $p$ is not a {\em bad prime} for $G$,
\item $p\nmid\abs{\pi_0(G)}$.
\end{enumerate}
The first assumption is needed in Hakim--Murnaghan's work. The
second one is used for the definition of regular supercuspidal
representations and also to ensure the existence of Howe
factorizations of tame regular elliptic pairs, see Remark \ref{rem.
for assumption}. The third assumption is required for the proof of
Lemma \ref{lem. key lemma} and for the construction of regular
supercuspidal $L$-packets. The last one is also for the construction
of regular supercuspidal $L$-packets. We refer to \cite[\S2.1]{kal}
for more detailed explanations and discussions on the roles that
these assumptions play in his theory of regular supercuspidal
representations, and also for a brief summarization of bad primes
determined by the type of $G$.

\subsection*{Notation and convention}
Let $F$ be a non-archimedean local field as before, $O_F$ the ring of integers of
$F$, and $k_F$ the residue field of $F$. We fix a separable closure
$F^s$ of $F$ and denote by $\Gamma$ the  Galois group
$\Gal(F^s/F)$. We write $W_F$ for the Weil group of $F$, $I_F$
for the inertia subgroup of $W_F$, and $P_F$ for the tame inertia
subgroup of $I_F$. Let $F^u$ be the maximal unramified extension of
$F$ in $F^s$ with residue field $\ov{k}_F$.

For a connected reductive group $G$ defined over $F$, we denote by
$Z(G)$ its center, by $A(G)$ its split central torus, by $G_\der$ its derived subgroup, by $G_\ad$ the
adjoint quotient of $G_\der$, and by $\fg$ the Lie algebra of $G$.
For an element $g\in G$ we will write $\Ad(g)$ for the conjugation
action of $g$ on $G$, i.e., $\Ad(g)(x)=gxg^{-1}$ for $x\in G$, and
also for the adjoint action of $g$ on $\fg$. When we mention a
subgroup of $G$, we always assume that it is a closed algebraic
subgroup defined over $F$. For a subgroup $M$ of $G$, we use
$M^\circ$ to denote its identity connected component. For any subset
$U$ of $G$, we use $C_G(U)$ to denote the identity component of its
centralizer in $G$.

For an involution $\theta$ of $G$, we always mean that it is a non-trivial automorphism of order two and defined over $F$. We denote by
$G^\theta$ the $\theta$-fixed subgroup of $G$ and by
$G^{\theta,\circ}$ its identity component. Then both $G^\theta$ and
$G^{\theta,\circ}$ are reductive subgroups of $G$. The group $G(F)$
has a natural action on the set of involutions, which is given by
$$g\cdot\theta:=\Ad(g)\circ\theta\circ\Ad(g^{-1}).$$ Let $M$ be a subgroup of $G$
and $\phi$ a character of $M(F)$. For $g\in G(F)$ we denote
$${^gM}:=g^{-1}Mg\quad\textrm{and}\quad{^g\phi}:=\phi\circ\Ad(g),$$ where ${^g\phi}$ is a
character of ${^gM(F)}$. We will use the following fact frequently.
If $M$ is $g\cdot\theta$-stable then $^gM$ is $\theta$-stable and
$({^gM})^\theta={^g(M^{g\cdot\theta})}$. For a $\theta$-stable
subgroup $U$ of $G(F)$, let $$U^{1+\theta}:=\{u\theta(u)\ |\ u\in U\}\subseteq U$$  denote the subgroup
of norms with respect to $\theta$. If $(\pi,V_\pi)$ is a representation of $G(F)$
where $V_\pi$ is the underlying space of $\pi$, we use
$\pi\circ\theta$ to denote the representation of $G(F)$ with
underlying space $V_\pi$ and action given by
$(\pi\circ\theta)(g)v=\pi(\theta(g))v$ for $v\in V_\pi$.

We will use similar notation as above when we discuss objects over
finite fields.

For a maximal torus $S$ of $G$, we denote by $N(S,G)$ the normalizer
of $S$ in $G$, by $\Omega(S,G)=N(S,G)/S$ the absolute Weyl group,
and by $R(S,G)$ the corresponding set of roots. The Galois
group $\Gamma$ has a natural action on $R(S,G)$. For any $\alpha\in
R(S,G)$, we denote by $\Gamma_\alpha$ (resp. ${\Gamma_{\pm\alpha}}$)
the stabilizer of $\alpha$ (resp. $\{\alpha,-\alpha\}$) in $\Gamma$,
and by $F_\alpha$ (resp. $F_{\pm\alpha}$) the corresponding fixed
subfield of $F^s$. We call $\alpha$ symmetric if the degree of
the extension $F_\alpha/F_{\pm\alpha}$ is 2, and call asymmetric
otherwise. We call $\alpha$ ramified or unramified if the extension
$F_\alpha/F_{\pm\alpha}$ is ramified or unramified respectively.

We denote by $\CB^\red(G,F)$ the reduced Bruhat--Tits building of
$G(F)$, and by $\CA^\red(S,F)$ the reduced apartment of $S$ in
$\CB^\red(G,F)$ where $S$ is a maximal torus of $G$ that is
maximally split. For $x\in\CB^\red(G,F)$, we write $G(F)_x$ for the stabilizer
of $x$ in $G(F)$, $G(F)_{x,0}$ for
the parahoric subgroup of $G(F)$ attached to $x$, $G(F)_{x,0+}$ for
its pro-unipotent radical, and $\sfG_x$ for the corresponding
connected reductive group over $k_F$. More generally, we denote by
$G(F)_{x,r}$ the Moy-Prasad filtration subgroups for any
$r\in\BR_{\geq0}$ and by $\fg(F)_{x,r}$ the filtration lattices of
$\fg(F)$ for any $r\in\BR$ (see \cite{mp94}). Moreover, we write
$G(F)_{x,r+}=\bigcup\limits_{s>r}G(F)_{x,s}$,
$G(F)_{x,r:s}=G(F)_{x,r}/G(F)_{x,s}$ and
$\fg(F)_{x,r:s}=\fg(F)_{x,r}/\fg(F)_{x,s}$ for $s>r$. We use
$\tilde{\BR}$ to denote the set $\BR\cup\{r+|
r\in\BR\}\cup\{\infty\}$.

Given a torus $T$ defined over $F$, let $\fT$ be the connected Neron
model of $T$ over $O_F$. We denote by $T(F)_0$ the subgroup
$\fT(O_F)$ of $T(F)$. We write $T^u$ for the maximal unramfied
subtorus of $T$. We can also define the Moy-Prasad filtration
subgroups $T(F)_r$ for any $r\geq 0$. In particular, when
$T=\Res_{E/F}\BG_m$, we have $E_0^\times=O_E^\times$ and
$E_r^\times=1+\fp_E^{[er]}$ for $r>0$, where $\fp_E$ is the maximal
ideal of $O_E$ and $e$ is the ramification index of the finite
extension $E/F$.

\subsection*{Acknowledgements}
This work was partially supported by NSFC Grants (11971223, 11501033), Fundamental Research Funds for the Central Universities (14380018), and Zheng Gang Scholars Program. Part of this
work was completed while the author was a visitor of Yau
Mathematical Sciences Center at Tsinghua University; he thanks the
institution for the hospitality. He would like to thank Tasho
Kaletha and  Cheng-Chiang Tsai for kindly answering several
questions, and thank Jeffrey Hakim for communicating his recent work
and helpful discussions. He also thanks the anonymous referees for pointing out several mathematical inaccuracies and many helpful suggestions to improve the paper.

\section{Preliminaries}\label{sec. preliminaries}
\subsection{Yu's construction}
In this subsection we briefly review Yu's construction of tame
supercuspidal representations \cite{yu01}. Let $G$ be a tamely ramified connected reductive group over $F$.

\subsubsection{Cuspidal $G$-data}\label{subsubsec. cuspidal G-datum}
Recall that a {\em cuspidal
$G$-datum} is a 4-tuple $\Psi=(\vec{G},x,\rho,\vec{\phi})$ that
satisfies the following conditions:
\begin{enumerate}
\item $\vec{G}$ is a tamely ramified twisted Levi sequence
$\vec{G}=(G^0,...,G^d)$ in $G$ such that $Z(G^0)/Z(G)$ is
anisotropic.
\item $x$ is a point in $\CA^\red(S,F)$, where $S$ is a tame maximal
torus of $G^0$.
\item $\rho$ is an irreducible representation of $K^0:=G^0(F)_x$ such
that $\rho|_{G^0(F)_{x,0+}}$ is {\bf1}-isotypic and the compactly
induced representation $\pi_{-1}=\ind_{K^0}^{G^0(F)}(\rho)$ is
irreducible.
\item $\vec{\phi}=(\phi_0,...,\phi_d)$ is a sequence of
quasicharacters, where $\phi_i$ is a quasicharacter of $G^i(F)$. We
require that: if $d=0$ then $\phi_0$ is of depth $r_0\geq0$; if
$d>0$ and $\phi_d$ is non-trivial then $\phi_i$ is of depth $r_i$ for
$i=0,...,d$ and $0<r_0<r_1<\cdots<r_{d-1}<r_d$; if $d>0$ and
$\phi_d$ is trivial then $\phi_i$ is of depth $r_i$ for
$i=0,...,d-1$ and $0<r_0<r_1<\cdots<r_{d-1}$. We will call
$\vec{r}=(r_0,...,r_d)$ the depth of $\vec{\phi}$ for short, and denote
$\phi:=\prod\limits_{i=0}^d\phi_i|_{G^0(F)}$.
\end{enumerate}

Note that the condition on $\rho$ implies that $\pi_{-1}$ is
supercuspidal and of depth zero. Conversely every irreducible
depth-zero supercuspidal representation of $G^0(F)$ arises in this
way. We call a triple $\Psi=(\vec{G},\pi_{-1},\vec{\phi})$ a {\em
reduced cuspidal $G$-datum} if $\vec{G}$ and $\vec{\phi}$ satisfy
the condition 1 and 4 above respectively and $\pi_{-1}$ is an
irreducible depth-zero supercuspidal representation of $G^0(F)$.
There is no essential difference between cuspidal $G$-datum and
reduced cuspidal $G$-datum.

We say that a (reduced) cuspidal $G$-datum is {\em generic} if
$\phi_i$ is $G^{i+1}$-generic for $i\neq d$. We refer to
\cite[Definition 3.9]{hm08} for the notion of genericity. Throughout
this article, we will only deal with generic (reduced) cuspidal
$G$-data, and will call them {\em $G$-data} for short.

\subsubsection{The representation $\pi(\Psi)$}\label{subsubsec. the representation pi}
Let $\Psi$ be a cuspidal $G$-datum. Let $K^0=G^0(F)_{x}$ and
$K^0_+=G^0(F)_{x,0+}$. For $0\leq i\leq d-1$, set
$s_i=\frac{r_i}{2}$, and put
$$K^{i+1}=K^0G^1(F)_{x,s_0}\cdots G^{i+1}(F)_{x,s_i},\quad
\ K_+^{i+1}=K^0_+G^1(F)_{x,s_0+}\cdots G^{i+1}(F)_{x,s_i+}.$$ 

When $\Psi$ is generic, Yu obtained an irreducible supercuspidal
representation $\pi(\Psi)$ of $G(F)$, called {\em tame supercuspidal
representation}, by a very technical process. The basic idea is
first to construct a representation $\kappa$ of $K^d$ from $\rho$
and generic quasicharacters $\phi_i$, and then set
$\pi(\Psi)=\ind_{K^d}^{G(F)}\kappa$. We refer to \cite{yu01} for
more details. In summary we have a map from the set of $G$-data to
that of tame supercuspidal representations.

\subsubsection{$G$-equivalence}\label{subsubsec. G-equivalence}
To study the dependence of $\pi(\Psi)$ on $\Psi$, Hakim and
Murnaghan  introduced three operations, which are called {\em
refactorization}, {\em elementary transformation} and {\em
$G$-conjugation}, on generic (reduced) cuspidal $G$-data. We refer
to \cite[Definition 4.19]{hm08}, \cite[Definitions 5.2 and 6.2]{hm08}
and \cite[page 110]{hm08} for the definition of these three
operations respectively. Note that these operations do not change the
genericity. Two $G$-data are called {\em $G$-equivalent} if they can
be obtained from each other by a finite sequence of these three
operations. One of the main result of \cite{hm08} is the following.

\begin{thm}\label{thm. HM G-equivalent}
Let $\Psi$ and $\dot{\Psi}$ be two generic (reduced) cuspidal
$G$-data. Then $\pi(\Psi)$ and $\pi(\dot{\Psi})$ are isomorphic if
and only if $\Psi$ and $\dot{\Psi}$ are $G$-equivalent.
\end{thm}

\begin{remark}\label{rem. hypothesis C(G)}
Theorem \ref{thm. HM G-equivalent} was proved in {\em loc. cit.} Theorem
6.6 under a hypothesis called $C(\vec{G})$. We refer
to {\em loc. cit.} page 47 for the precise statement of $C(\vec{G})$.
Recently this hypothesis was removed by Kaletha \cite[Corollary
3.5.5]{kal}.
\end{remark}

\subsection{Kaletha's work}

Kaletha's recent work \cite{kal} provides a more elegant
parametrization for most of the tame supercuspidal representations.
These representations are called regular supercuspidal
representations, which are the main objects of our paper. As before, $G$ is a tamely ramified connected reductive group over $F$.

\subsubsection{Maximally unramified elliptic tori and regular depth-zero characters}
A maximal torus $S$ of
$G$ is called {\em maximally unramified} if $S\times F^u$ is a
minimal Levi subgroup of the quasi-split group $G\times F^u$. See
\cite[Fact 3.4.1]{kal} for other equivalent definitions. Now let $S$
be a maximally unramified elliptic maximal torus. Recall that we
denote by $S^u$ the maximal unramified subtorus of $S$. The unique
Frobenius-fixed point $x$ in $\CA^\red(S,F^u)$ is a vertex of
$\CB^\red(G,F)$ \cite[Lemma 3.4.3]{kal}. We have (\cite[Lemmas 3.1.6 and 3.4.6]{kal})
$$S^u(F)_0/S^u(F)_{0+}\stackrel{\simeq}{\lra}S(F)_0/S(F)_{0+},\quad S(F)_0=S(F)\cap G(F)_{x,0}.$$
The special fiber of the connected Neron model of $S^u$ embeds in $\mathsf{G}_x$ as an elliptic maximal torus $\mathsf{S^u}$, and the images of $S(F)_0$ and $S^u(F)_0$ in $\mathsf{G}_x(k_F)$ are
equal to $\mathsf{S^u}(k_F)$. Conversely, for a vertex  $x$ of $\CB^\red(G,F)$,
every elliptic maximal torus of $\mathsf{G}_x$ arises in this way. See {\em loc. cit.} Lemma 3.4.4.

For a vertex $x$ of $\CB^\red(G,F)$ and an elliptic maximal torus $\mathsf{S^u}$ of $\mathsf{G}_x$ that corresponds to a maximally unramified elliptic maximal torus $S$ of $G$, a character $\lambda$ of $\sfS^\sfu(k_F)$ is called {\em regular} if its stabilizer in $N(S,G)(F)/S(F)$ is trivial. A depth-zero character $\mu$ of $S(F)$ is called {\em regular} if $\mu|_{S(F)_0}$ induces a regular character $\bar{\mu}$ of $\sfS^\sfu(k_F)$; see {\em loc. cit.} Definition 3.4.16. Note that if $\lambda$ is a regular character of $\sfS^\sfu(k_F)$ then it is in general position which means that its stabilizer in $\Omega(\sfS^\sfu,\sfG_x)(k_F)$ is trivial (cf. {\em loc. cit.} Fact 3.4.18).

\subsubsection{Regular supercuspidal representations}\label{subsubsec. regular sc repns}
Let $\Psi=(\vec{G},\pi_{-1},\vec{\phi})$ be a reduced generic
cuspidal $G$-datum. Recall that $\pi_{-1}$ is an irreducible
depth-zero supercuspidal representation of $G^0(F)$. By
\cite[Proposition 6.8]{mp96}, there exists a vertex
$x\in\CB^\red(G^0,F)$ such that $\pi_{-1}|_{G^0(F)_{x,0}}$ contains
the inflation to $G^0(F)_{x,0}$ of an irreducible cuspidal
representation $\kappa$ of
$G^0(F)_{x,0:0+}\simeq\mathsf{G^0}_x(k_F)$. 

We call $\Psi$ {\em
regular} if $\kappa$ is a Deligne--Lusztig cuspidal representation
$\pm R_{\mathsf{T},\lambda}$ attached to an elliptic maximal torus
$\mathsf{T}$ of $\mathsf{G^0}_x$ and a regular character $\lambda$ of
$\mathsf{T}(k_F)$ (cf. \cite[Definition 3.7.9]{kal}). Note that if $\Psi$ is
regular, then any $G$-datum in its $G$-equivalent class is also
regular. We call an irreducible supercuspidal representation $\pi$
of $G(F)$ {\em regular} if it is of the form $\pi(\Psi)$ for some
regular generic reduced cuspidal $G$-datum $\Psi$. According to
Theorem \ref{thm. HM G-equivalent} the regularity of $\pi$ is well
defined.

\begin{remark}\label{rem. for assumption}
The above definition of regular supercuspidal representations requires
the assumption that $p\nmid \abs{\pi_1(G_\der)}$. More generally, if $p$ divides $\abs{\pi_1(G_\der)}$, an
irreducible supercuspidal representation $\pi$ of $G(F)$ is called
regular if its inflation to $\tilde{G}(F)$ is regular in the above sense,
where $\tilde{G}\ra G$ is a $z$-extension. See \cite[\S3.7.4]{kal} for more details. In this situation,
there may exist regular supercuspidal representations which can not
be constructed from Yu's construction. One reason that we need the
assumption $p\nmid |\pi_1(G_\der)|$ is that we want to apply Hakim--Murnaghan's result that is
valid for tame supercuspidal representations. Another reason is that
we need the existence of Howe factorizations of tame regular
elliptic pairs, see Section \ref{subsubsec. tame regular elliptic
pairs}.
\end{remark}

\subsubsection{Parametrization: depth-zero case}\label{subsubsec. tame regular pairs depth-zero
case} 
Suppose that $\pi$ is a regular depth-zero supercuspidal
representation of $G(F)$. By \cite[Proposition 3.4.27]{kal} there exists a
maximally unramified elliptic maximal torus $S$ of $G$ and a
regular depth-zero character $\mu$ of $S(F)$ such that $\pi$
is of the form $\pi_{(S,\mu)}$. 

First let us  recall the definition of $\pi_{(S,\mu)}$. Let $x$ be the vertex determined by $S$, 
and $\pm R_{\mathsf{S^u},\bar{\mu}}$ the Deligne--Lusztig
cuspidal representation of $\mathsf{G}_x(k_F)$ associated with
$\mathsf{S^u}$ and $\bar{\mu}$. We denote by $\kappa_{(S,\mu)}$ the
inflation of $\pm R_{\mathsf{S^u},\bar{\mu}}$ to $G(F)_{x,0}$. In \cite[\S3.4.4]{kal}, Kaletha constructed a representation $\tilde{\kappa}_{(S,\mu)}$ of $G_S:=S(F)G(F)_{x,0}$, which is an extension of $\kappa_{(S,\mu)}$. The technical issue is that in general $Z(F)
S(F)_0$ is not equal to $S(F)$, which makes the construction of
$\tilde{\kappa}_{(S,\mu)}$ more subtle. According to {\em loc. cit.} Lemma
3.4.20, the representation
$$\rho_{(S,\mu)}:=\ind_{G_S}^{G(F)_x}\tilde{\kappa}_{(S,\mu)}$$ is
irreducible and thus the representation
$$\pi_{(S,\mu)}:=\ind_{G(F)_x}^{G(F)}\rho_{(S,\mu)}$$ is a regular
depth-zero supercuspidal representation. 

Moreover, \cite[Proposition 3.4.27]{kal} states that the isomorphism classes of regular depth-zero supercuspidal representations are parameterized by the conjugacy classes of the pairs $(S,\mu)$. In summary, we have:

\begin{lem}\label{lem. parameterize regular depth-zero repn}
Each regular depth-zero supercuspidal representation is of the form $\pi_{(S,\mu)}$ for some pair $(S,\mu)$ as above. Two regular depth-zero supercuspidal representations
$\pi_{(S_1,\mu_1)}$ and $\pi_{(S_2,\mu_2)}$ are isomorphic if and
only if the pairs $(S_1,\mu_1)$ and $(S_2,\mu_2)$ are
$G(F)$-conjugate.
\end{lem}

\subsubsection{Parametrization: general case}\label{subsubsec. tame regular elliptic pairs}
To obtain an analogous parametrization as Lemma \ref{lem.
parameterize regular depth-zero repn} for general regular
supercuspidal representations, Kaletha introduced the notion {\em
tame regular elliptic pairs} $(S,\mu)$, where $S$ is a tame elliptic
maximal torus of $G$ and $\mu$ a character of $S(F)$ satisfying the
conditions in \cite[Definition 3.7.5]{kal}.

Suppose that $\Psi=(\vec{G},\pi_{(S,\mu_\circ)},\vec{\phi})$ is a
regular reduced generic cuspidal $G$-datum, where $S$ is a maximally
unramified elliptic maximal torus of $G^0$ and $\mu_\circ$  a
regular depth-zero character of $S(F)$ with respect to $G^0$. Then
$(S,\mu)$ is a tame regular elliptic pair \cite[Proposition 3.7.8]{kal},
where $$\mu=\mu_\circ\prod_{i=0}^d\phi_i|_{S(F)}.$$

Conversely, given a tame regular elliptic pair $(S,\mu)$, a Howe
factorization \cite[\S3.6]{kal} of $(S,\mu)$ provides a regular
generic cuspidal $G$-datum $\Psi$ and thus a regular supercuspidal
representation $$\pi_{(S,\mu)}:=\pi(\Psi)$$ of $G(F)$. For later use, let us
review the definition of Howe factorization. Let $E$ be the
splitting field of $S$. For each $r>0$, the Levi subsystem
$$R_r=\{\alpha\in R(S,G)\ |\ \mu(\N_{E/F}(\alpha^\vee(E_r^\times)))=1\}$$
of $R(S,G)$ gives a filtration $r\mapsto R_r$ of $R(S,G)$. The Levi subsystems $R_{r_i}$'s associated to the 
breaks $r_{d-1}>r_{d-2}>\cdots>r_0>0$ in the filtration of $R(S,G)$ give rise to a twisted Levi
sequence
\begin{equation}\label{equ. twisted levi}\vec{G}=(G^0,\cdots,G^d=G),\end{equation} where, for $0\leq i\leq d-1$, $G^i$ is the twisted Levi
subgroup of $G$ such that $S$ is a maximal torus and
$R(S,G^i)=R_{r_i}$. Set $r_d=\depth(\mu)$ and
\begin{equation}\label{equ. breaks}\vec{r}=(r_0,...,r_d).\end{equation}
A {\em Howe factorization} of $(S,\mu)$ is
a sequence of characters:
$$\mu_\circ:S(F)\ra\BC^\times,\quad\vec{\phi}=\left(\phi_i: G^i(F)\ra\BC^\times,
\ 0\leq i\leq d\right)$$ satisfying the conditions: $\mu_\circ$ is
regular (with respect to $G^0$) and of depth zero, $\vec{\phi}$ is
of depth $\vec{r}$ such that
$(\vec{G},\pi_{(S,\mu_\circ)},\vec{\phi})$ is a {\em normalized} (cf. \cite[Definition 3.7.1]{kal}) regular
reduced cuspidal generic $G$-datum and
$\mu=\mu_\circ\prod\limits_{i=0}^d\phi_i|_{S(F)}$. 
For convenience, we also call the above $G$-datum
$(\vec{G},\pi_{(S,\mu_\circ)},\vec{\phi})$ a Howe factorization of
$(S,\mu)$. Under the assumption that $p\nmid\abs{\pi_1(G_\der)}$,
Howe factorizations always exist and differ by refactorizations; see
{\em loc. cit.} Proposition 3.6.7 and Lemma 3.6.6. We remark that a Howe factorization is defined for any pair $(S,\mu)$ where $S$ is a tame maximal torus of $G$ and $\mu$ an arbitrary character of $S(F)$ in {\em loc. cit.} Definition 3.6.2. We refer to \cite[\S3.6]{kal} for more details.

In summary, 
\cite[Proposition 3.7.8 and Corollary 3.7.10]{kal} establish a parametrization of
regular supercuspidal representations in terms of tame regular
elliptic pairs.

\begin{lem}\label{lem. parameterize regular repn}
Each regular supercuspidal representation is of the form $\pi_{(S,\mu)}$ for some tame regular elliptic pair $(S,\mu)$.	
Two regular supercuspidal representations $\pi_{(S_1,\mu_1)}$ and
$\pi_{(S_2,\mu_2)}$ are isomorphic if and only if the pairs
$(S_1,\mu_1)$ and $(S_2,\mu_2)$ are $G(F)$-conjugate.
\end{lem}

\subsection{Hakim--Murnaghan's work}

Let $\theta$ be an involution of $G$  and $H=G^\theta$. Let
$\Psi=(\vec{G},x,\rho,\vec{\phi})$ be a generic cuspidal $G$-datum
and $\pi=\pi(\Psi)$ the irreducible supercuspidal representation of
$G(F)$ attached to $\Psi$ (see \S2.1.1). Hakim--Murnaghan's result \cite[Theorem
5.26]{hm08} provides an explicit formula for
$\dim\Hom_{H(F)}(\pi,{\bf1})$. Later Hakim and Lansky \cite{hj12} corrected
some mistakes in \cite{hm08} and improved the theory.
\begin{defn}\label{defn. theta symmetric}
We say that $\Psi$ is {\em $\theta$-symmetric} if
\begin{itemize}
\item $\theta(\vec{G})=\vec{G}$, i.e., $\theta(G^i)=G^i$ for any $0\leq i\leq
d$,
\item $\vec{\phi}\circ\theta=\vec{\phi}^{-1}$, i.e.,
$\phi_i\circ\theta=\phi_i^{-1}$ for any $0\leq i\leq d$,
\item $\theta(x)=x$.
\end{itemize}
\end{defn}

As in \cite{hj12}, we denote by $[\Psi]$ the set of refactorizations of $\Psi$ and by
$[\theta]$ the $K^0$-orbit of $\theta$. Recall that we denote $\phi=\prod\limits_{i=0}^d\phi_i|_{G^0(F)}$. Set
$\rho_\nm:=\rho\otimes\left(\phi|_{K^0}\right)$, which is an
invariant of $[\Psi]$. Note that $\phi|_{K^0_+}$ is also an
invariant of $[\Psi]$.

\begin{defn}\label{defn. theta symmetric 2}
We write $[\theta]\sim[\Psi]$ if
\begin{equation}\label{equ. weak theta symmetric}\theta(K^0)=K^0\quad \textrm{and}\quad
\phi|_{K^{0,\theta}_+}=1.\end{equation}
\end{defn}

We remark that $[\theta]\sim[\Psi]$ is well defined since the
condition (\ref{equ. weak theta symmetric})  depends only on
$[\theta]$ and $[\Psi]$. The relation between Definition \ref{defn.
theta symmetric} and Definition \ref{defn. theta symmetric 2} is as follows \cite[Proposition 3.9]{hj12}.

\begin{lem}\label{lem. relation theta symmetric}
We have $[\theta]\sim[\Psi]$ if and only if there exists
$\dot{\Psi}\in[\Psi]$ such that $\dot{\Psi}$ is $\theta$-symmetric.
\end{lem}

In particular $[\theta]\sim[\Psi]$ implies that
$\theta(\vec{G})=\vec{G}$ and $\theta(x)=x$. For our purpose on the
distinction problem, we only need the following partial result
derived from \cite[Theorem 5.26]{hm08} and \cite[Theorem
3.10]{hj12}.

\begin{thm}\label{thm. HM distinction} For a $G$-datum $\dot{\Psi}$,
the representation $\pi(\dot{\Psi})$ is $H$-distinguished if and
only if $\dot{\Psi}$ is $G$-equivalent to a $G$-datum $\Psi$ such
that
\begin{enumerate}
\item $[\theta]\sim[\Psi]$,
\item
$\Hom_{K^{0,\theta}}(\rho_\nm,\eta_{\theta})\neq0$.
\end{enumerate}
\end{thm}

Here $\eta_{\theta}$ is a quadratic character of $K^{0,\theta}$,
which is introduced in \cite[\S5.6]{hm08} and denoted by $\eta'_\theta$ therein. We caution the reader that in \S5.6 of {\em loc. cit.} there is another character denoted by $\eta_\theta$. We apologize for the inconsistency of this notation. Now let us record the definition of $\eta_\theta$. For each $0\leq i\leq d-1$, consider the subgroups
$$J^{i+1}=(G^i,G^{i+1})(F)_{x,(r_i,s_i)}\quad\textrm{and}\quad J^{i+1}_+=(G^i,G^{i+1})(F)_{x,(r_i,s_i+)}$$ of $G^{i+1}(F)_{x,s_i}$ and $G^{i+1}(F)_{x,s_i+}$ respectively,
whose definitions are given at the end of \cite[\S3]{yu01} (also see \cite[page 53]{hm08}).
The quotient group
$$W_i:=J^{i+1}/J^{i+1}_+$$ is equipped with a structure of symplectic
$\BF_p$-vector space \cite[Lemma 11.1]{yu01}. Since $[\theta]\sim
[\Psi]$, both $J^{i+1}$ and $J^{i+1}_+$ are $\theta$-stable for each
$i$. Thus $\theta$ induces a linear transformation on $W_i$, which
is still denoted by $\theta$. Set
$$W_i^{\theta}=\left\{w\in W_i | \theta(w)=w\right\}.$$ Then $W_i^{\theta}$ is stable by
$K^{0,\theta}$ under the conjugate action. Let $\chi_i^\theta$  be
the quadratic character of $K^{0,\theta}$ defined by
\begin{equation}\label{equ. character chi theta}
\chi_i^\theta(k)=\det\left(\Ad(k)|_{W_i^{\theta}}\right)^{\frac{p-1}{2}}.\end{equation}
Then the character $\eta_\theta$ is defined to be
\begin{equation}\label{equ. character eta_theta}
\eta_\theta=\prod_{i=0}^{d-1}\chi_i^\theta,\end{equation} and trivial if $d=0$. Note that
$\eta_\theta$  depends only on $[\Psi]$.

\section{Distinction}

\subsection{Main results}
Our main theorem is Theorem \ref{thm. main theorem}, whose statement
is given in Section \ref{subsubsec. main theorems} and whose proof
is delayed to Section \ref{subsec. proofs}. We first introduce two
characters $\varepsilon_S$ and $\eta_S$, which are involved in
Theorem \ref{thm. main theorem}, in Sections \ref{subsubsec.
character varepsilon} and \ref{subsubsec. character eta}
respectively. A direct but important consequence (Corollary
\ref{cor. cor of main theorem}) of the main theorem is also stated
in Section \ref{subsubsec. main theorems}. Some examples are
discussed in Section \ref{subsec. supplements}.

As before, we always assume that $G$ is a tamely ramified connected
reductive group over $F$ and $\theta$ an involution of $G$.

\subsubsection{The character $\varepsilon_S$}\label{subsubsec. character
varepsilon} First we review the character $\varepsilon_\mathsf{T}$
introduced in \cite[\S2]{lus90} for the distinction problem over finite fields. Let
$\mathsf{G}$ be a connected reductive group over $k_F$ and $\theta$
an involution of $\mathsf{G}$ defined over $k_F$. Suppose that
$\mathsf{T}$ is a $\theta$-stable maximal $k_F$-torus of
$\mathsf{G}$. Recall that we denote by $\mathsf{T}^{\theta,\circ}$
the identity component of $\mathsf{T}^\theta$. As a map on $\mathsf{T}^\theta(k_F)$,
$\varepsilon_\mathsf{T}$ is given  by
$$\varepsilon_{\mathsf{T}}(t)=\sigma(C_{\mathsf{G}}(\mathsf{T}^{\theta,\circ}))\cdot
\sigma(C_{\mathsf{G}}(\mathsf{T}^{\theta,\circ})\cap
C_\mathsf{G}(t)),$$ where
$\sigma(\mathsf{M})$ is set to be $(-1)^{\rank_{k_F}(\mathsf{M})}$ for any
connected reductive group $\mathsf{M}$ over $k_F$. By
\cite[Proposition 2.3 (b) and (c)]{lus90} the map
$\varepsilon_{\mathsf{T}}$ is actually a character, and satisfies
\begin{equation}\label{equ. character varepsilon relation}
\varepsilon_{\mathsf{T}}|_{\mathsf{T}^{\theta,\circ}(k_F)}=1.\end{equation}

Let $\mathsf{T}_\ad$ be the image of $\mathsf{T}$ in
$\mathsf{G}_\ad$. For $t\in\mathsf{T}_\ad$, we denote by
$C_\mathsf{G}(t)$ the identity component of the centralizer of $t$
in $\mathsf{G}$. The involution $\theta$ induces an involution,
still denoted by $\theta$, on $\mathsf{G}_\ad$. Let
$(\mathsf{T}_\ad)^{\theta,\circ}$ be the identity component of
$(\mathsf{T}_\ad)^\theta$, and $(\mathsf{T}^{\theta,\circ})_\ad$ the
image of $\mathsf{T}^{\theta,\circ}$ in $\mathsf{G}_\ad$.

For any $t\in(\mathsf{T}_\ad)^\theta(k_F)$, $C_\mathsf{G}(t)$ is
also defined over $k_F$. Therefore we can define a map, still denoted by
$\varepsilon_\mathsf{T}$, on $(\mathsf{T}_\ad)^\theta(k_F)$, which
is given by the same formula
$$\varepsilon_{\mathsf{T}}(t)=\sigma(C_{\mathsf{G}}(\mathsf{T}^{\theta,\circ}))\cdot
\sigma(C_{\mathsf{G}}(\mathsf{T}^{\theta,\circ})\cap
C_\mathsf{G}(t)).$$

\begin{lem}\label{lem. character varepsilon}
The map $\varepsilon_{\mathsf{T}}$ is a character of
$(\mathsf{T}_\ad)^{\theta}(k_F)$. Moreover we have
$$\varepsilon_{\mathsf{T}}|_{(\mathsf{T}_\ad)^{\theta,\circ}(k_F)}=1.$$
\end{lem}
\begin{proof}
Recall that, according to \cite[\S2.2]{ric82}, for an arbitrary torus $\sfS$ over $k_F$
equipped with an involution $\theta$, we have $\mathsf{S}^{\theta,\circ}=\{s\theta(s) | s\in\mathsf{S}\}$
over an algebraic closure $\bar{k}_F$ of $k_F$. Since
$(\mathsf{T}_\ad)^{\theta,\circ}(\bar{k}_F)=\{t\theta(t) |  t\in\mathsf{T}_\ad(\bar{k}_F)\}$
and $\mathsf{T}^{\theta,\circ}(\bar{k}_F)=\{t\theta(t) | t\in\mathsf{T}(\bar{k}_F)\}$, the natural embedding
$(\mathsf{T}^{\theta,\circ})_\ad\ra(\mathsf{T}_\ad)^{\theta,\circ}$
is also surjective and thus an isomorphism.   For
$t\in(\mathsf{T}_\ad)^{\theta,\circ}(k_F)=(\mathsf{T}^{\theta,\circ})_\ad(k_F)$,
take any lift $\dot{t}\in\mathsf{T}^{\theta,\circ}(\bar{k}_F)$. Then
$C_\mathsf{G}(t)=C_\mathsf{G}(\dot{t})\supset
C_\mathsf{G}(\mathsf{T}^{\theta,\circ})$, and thus
$\varepsilon_\mathsf{T}(t)=1$. The rest of the proof is same as that
of \cite[Proposition 2.3]{lus90}.
\end{proof}

Now we come back to the $p$-adic fields case. We first consider the
depth-zero case. Let $S$ be a maximally unramified elliptic maximal
torus of $G$ and $S^u$ the maximal unramified subtorus of $S$. Let
$x$ be the vertex of $\CB^\red(G,F)$ attached to $S$. Suppose that
$\theta(x)=x$. Thus both $G(F)_{x,0}$ and $G(F)_{x,0+}$ are
$\theta$-stable. Therefore $\theta$ induces an involution on
$\mathsf{G}_x$, which is still denoted by ${\theta}$. We assume that
$\mathsf{S^u}$ is $\theta$-stable, where $\mathsf{S^u}$ is the
elliptic maximal torus of $\mathsf{G}_x$ corresponding to $S^u$.

\begin{lem}\label{lem. stable torus}
There exists $y\in G(F)_{x,0+}$ such that $^yS$ is $\theta$-stable.
\end{lem}

\begin{proof} When $S$ is unramified, the assertion follows from
\cite[Lemma A.2]{hj12} and the results in \cite[\S2]{de06}. In
general, using the same proof as that of \cite[Lemma A.2]{hj12}, we
can show that there exists a $\theta$-stable maximally unramified
elliptic maximal torus $S_1$ of $G$ such that
$x\in\CA^\red(S_1,F^u)$ and $\mathsf{S^u}$ also corresponds to $S_1^u$.
According to \cite[Lemma 3.4.5]{kal}, $S_1$ and $S$ are
$G(F)_{x,0+}$-conjugate.
\end{proof}

Recall that we denote $G_S=S(F)G(F)_{x,0}$. We use the same notation
$\mathsf{S}(k_F)$ as \cite[\S3.4.4]{kal} to denote
$$\mathsf{S}(k_F):=S(F)/S(F)_{0+},$$
which is a subgroup of
$$\mathsf{G}_S(k_F):=G_S/G(F)_{x,0+}.$$

\begin{cor}\label{cor. stable torus}
Both $G_S$ and $\mathsf{S}(k_F)$ are $\theta$-stable.
\end{cor}
\begin{proof}
According to Lemma \ref{lem. stable torus}, there exists a
$\theta$-stable torus $S_1$ which is $G(F)_{x,0+}$-conjugate to $S$. Since $S(F)$ normalizes $G(F)_{x,0}$,
we see $G_S$ is also equal to $S_1(F)G(F)_{x,0}$ which is
$\theta$-stable. Moreover $\mathsf{S}(k_F)$ coincides with the image
of $S_1(F)$ in $\mathsf{G}_S(k_F)$, which is also $\theta$-stable.
\end{proof}

According to the third paragraph of \cite[\S3.4.4]{kal}, there is a natural homomorphism
$$\iota:\mathsf{S}(k_F)\ra\mathsf{S^u_\ad}(k_F)$$ where $\mathsf{S^u_{\ad}}$ is the
image of $\mathsf{S^u}$ in $[\mathsf{G}_x]_\ad$, and $\iota$ is
given by the composition
$$S(F)\ra S_\ad(F)=S_\ad(F)_0\ra
S_\ad(F)_{0:0+}=[\mathsf{S_\ad}]^\mathsf{u}(k_F)\ra\mathsf{S^u_\ad}(k_F)$$
where $S_\ad$ is the image of $S$ in $G_\ad$,
$[\mathsf{S_\ad}]^\mathsf{u}$ the elliptic maximal torus of
$[\mathsf{G}_\ad]_x$ corresponds to $(S_\ad)^u$, and
$[\mathsf{S}_\ad]^\mathsf{u}\ra\mathsf{S_\ad^u}$ given by the
natural map $[\mathsf{G}_\ad]_x\ra[\mathsf{G}_x]_\ad$.
Therefore the image of $\mathsf{S}(k_F)^\theta$ under $\iota$ lies
in $(\mathsf{S^{u}_\ad})^\theta(k_F)$. The character
$\varepsilon_\mathsf{S}$ of $\mathsf{S}(k_F)^\theta$ is defined to
be
\begin{equation}\label{equ. chracter varepsilon}
\varepsilon_\mathsf{S}=\varepsilon_{\mathsf{S^u}}\circ\iota.\end{equation}

\begin{defn}\label{defn. character varepsilon}
Suppose that $S$ is $\theta$-stable. The {\em character
$\varepsilon_S$} is defined to be the composition of the natural map
$S^\theta(F)\ra\mathsf{S}(k_F)^\theta$ and $\varepsilon_\mathsf{S}$.
\end{defn}

\begin{remark}
For general depth case, suppose that $(S,\mu)$ is a tame regular
elliptic pair of $G$ such that $S$ is $\theta$-stable and $\mu|_{S(F)^\theta_{0+}}=1$. Let $G^0$ be
the 0th twisted Levi subgroup in the twisted Levi sequence $\vec{G}$
of $G$ determined by $(S,\mu)$. Then $G^0$ is also $\theta$-stable
(see Lemma \ref{lem. stable Levi sequence} below). Recall that $S$
is a maximally unramified elliptic maximal torus of $G^0$. We define
the character $\varepsilon_S$ of $S^\theta(F)$ as in Definition
\ref{defn. character varepsilon}, but with respect to $G^0$.
\end{remark}

\begin{lem}\label{lem. character varepsilon 2}
Suppose that $S$ is $\theta$-stable and $\mu|_{S(F)^\theta_{0+}}=1$. Then we have
$$\varepsilon_S|_{S^{\theta,\circ}(F)}=1.$$
\end{lem}
\begin{proof}
It is obvious that the image of $S^{\theta,\circ}(F)$ in
$(\mathsf{S^u_\ad})^\theta(k_F)$ is actually in
$(\mathsf{S^u_\ad})^{\theta,\circ}(k_F)$. Hence the assertion
follows from Lemma \ref{lem. character varepsilon} immediately.
\end{proof}

\subsubsection{The character $\eta_S$}\label{subsubsec. character
eta}

\begin{lem}\label{lem. stable Levi sequence}
Suppose that $(S,\mu)$ is a tame regular elliptic pair of $G$. Let
$\vec{G}$ be the twisted Levi sequence determined by $(S,\mu)$, and
$x$ the vertex of $\CB^\red(G^0,F)$ attached to $S$. If $S$ is
$\theta$-stable and $\mu|_{S(F)^\theta_{0+}}=1$, then $\theta(\vec{G})=\vec{G}$ and $\theta(x)=x$.
\end{lem}

\begin{proof}
Since $S$ is $\theta$-stable, $\theta$ acts on $R(S,G)$ by
$\theta(\alpha):=\alpha\circ\theta$ and acts on $R(S,G)^\vee$ by
$\theta(\alpha^\vee)=\theta\circ\alpha^\vee$. It is clear that
$\theta(\alpha^\vee)=\theta(\alpha)^\vee$. If $\alpha\in R_r$, that
is $\mu(\N_{E/F}(\alpha^\vee(E^\times_r)))=1$, we have
$$\begin{aligned}
\mu(\N_{E/F}(\theta(\alpha)^\vee(E^\times_r)))
=&\mu(\N_{E/F}(\theta\circ\alpha^\vee(E^\times_r)))\\
=&\mu(\theta\circ\N_{E/F}(\alpha^\vee(E^\times_r)))\\
=&\mu(\N_{E/F}(\alpha^\vee(E^\times_r)))^{-1}\\
=&1.
\end{aligned}$$
Hence $R_r$ is $\theta$-stable. Therefore each twisted Levi subgroup
$G^i$ is $\theta$-stable. It is obvious that $\theta(x)=x$.
\end{proof}

Now we assume that $(S,\mu)$ is a tame regular elliptic pair of $G$ such that $S$
is $\theta$-stable and $\mu|_{S(F)^\theta_{0+}}=1$. Let $\vec{G}$ and $\vec{r}$ be the twisted Levi sequence and the sequence of depths determined by $(S,\mu)$; see equations (\ref{equ. twisted levi}) and (\ref{equ. breaks}) in \S\ref{subsubsec. tame regular elliptic pairs}. Let $x$ be the vertex of $\CB^\red(G^0,F)$ attached to $S$. According to Lemma \ref{lem. stable Levi sequence}, it
makes sense to let $\eta_\theta$ be the character of $K^{0,\theta}$
defined by the formula (\ref{equ. character eta_theta}), where $K^0=G^0(F)_x$.
Note that $S(F)\subseteq K^0=G^0(F)_x$.
\begin{defn}\label{defn. character eta}
Under the above conditions, the {\em character $\eta_S$}
of $S^\theta(F)$ is defined to be
$$\eta_S=\eta_\theta|_{S^\theta(F)}.$$
\end{defn}

\begin{lem}\label{lem. key lemma}
Let $(S,\mu)$ be a tame regular elliptic pair such that $S$ is
$\theta$-stable and $\mu|_{S(F)^\theta_{0+}}=1$. Let $S^u$ be the maximal unramified subtorus of
$S$. Then we have
$$\eta_S|_{(S^u)^{\theta,\circ}(F)}=1.$$
\end{lem}

\begin{proof}
Let $\vec{G}$, $\vec{r}$ and $x$ be the above-mentioned objects determined by $(S,\mu)$.
Recall that the character $\eta_\theta$ is defined to be
$\prod\limits_{i=0}^{d-1}\chi_i^\theta$; see equation (\ref{equ. character eta_theta}). We will show that
$$\chi_i^\theta|_{(S^u)^{\theta,\circ}(F)}=1,\quad\forall\ 0\leq i\leq d-1.$$
To simplify the notation, we denote $r=r_i$, $s=\frac{r_i}{2}$,
$J=J^{i+1}$, $J_+=J^{i+1}_+$, $W=W_i$, $T=(S^u)^{\theta,\circ}$,
$G=G^{i+1}$, $G'=G^i$, $H=G^{\theta,\circ}$ and $H'=(G')^{\theta,\circ}$. Let
$\fg=\Lie(G)$, $\fg'=\Lie(G')$, $\fh=\Lie(H)$ and $\fh'=\Lie(H')$.
According to
Section \ref{subsubsec. assumption}, our assumptions on $p$ satisfy \cite[Hypothesis 2.1.1]{adl98}. Therefore there exists a $G(F)$-invariant non-degenerate symmetric
bilinear $F$-valued form $B$ on $\fg(F)$. Denote by $\fg'(F)^\bot$
the orthogonal complement of $\fg'(F)$ in $\fg(F)$ and by
$\fg'(F)^\bot_{x,t}$ the intersection $\fg'(F)^\bot\cap\fg(F)_{x,t}$
for any $t\in\tilde{\BR}$.  By \cite[Proposition 1.9.3]{adl98}, we
have $$\fg(F)_{x,t}=\fg'(F)_{x,t}\oplus\fg'(F)_{x,t}^\bot.$$ Put
$$\fJ=\fg'(F)_{x,r}\oplus\fg'(F)^\bot_{x,s},\quad
\fJ_+=\fg'(F)_{x,r}\oplus\fg'(F)^\bot_{x,s_+}.$$  There is a natural
$G'(F)_{x}$-equivariant isomorphism from $J/J_+$ to $\fJ/\fJ_+$.
Theorefore, for $k\in G'(F)^\theta_x$, we have
$$\begin{aligned}\chi^\theta(k)&=\det\left(\Ad(k)|_{W^\theta}\right)\\
&=\det\left(\Ad(k)|_{(J/J_+)^\theta}\right)\\
&=\det\left(\Ad(k)|_{(\fJ/\fJ_+)^\theta}\right).
\end{aligned}$$
Note that
$$\begin{aligned}
\fJ/\fJ_+&=\fg'(F)^\bot_{x,s}/\fg'(F)^\bot_{x,s_+}\\
&=\left(\fg'(F)_{x,s}\oplus\fg'(F)^\bot_{x,s}\right)/\left(\fg'(F)_{x,s}\oplus\fg'(F)^\bot_{x,s+}\right)\\
&=\left(\fg(F)_{x,s}/\fg(F)_{x,s+}\right)/\left((\fg'(F)_{x,s}\oplus\fg'(F)^\bot_{x,s+})/\fg(F)_{x,s+}\right)\\
&=\left(\fg(F)_{x,s}/\fg(F)_{x,s+}\right)/\left(\fg'(F)_{x,s}/\fg'(F)_{x,s+}\right)\\
&=\fg(F)_{x,s:s+}/\fg'(F)_{x,s:s+}.
\end{aligned}$$
Due to \cite[Lemma 2.11, Proposition 2.12]{hm08}, we can identify
$$\left(\fg(F)_{x,s:s+}/\fg'(F)_{x,s:s+}\right)^\theta=\fg(F)_{x,s:s+}^\theta/\fg'(F)_{x,s:s+}^\theta.$$
By \cite[Lemma 2.8]{por14}, for any $t\in\tilde{\BR}$, we have
$$\fg(F)^\theta_{x,t}=\fg(F)_{x,t}\cap\fh(F)=\fh(F)_{x,t},$$ and
$$\fg'(F)^\theta_{x,t}=\fg'(F)_{x,t}\cap\fh(F)=\fh'(F)_{x,t}.$$
Hence
$$\left(\fJ/\fJ_+\right)^\theta=\fh(F)_{x,s:s+}/\fh'(F)_{x,s:s+}.$$
Note that $T\subseteq H'\subseteq H$. Since $T$ is an unramified elliptic
torus, according to \cite[Lemma 7.1.1]{kal11} we have $T(F)=A(T)(F)\cdot T(F)_0$. Recall that $A(\cdot)$ denotes the central split torus of the corresponding reductive group. Since $S$ is an elliptic maximal torus in $G'$ and $G$, we have $A(S)=A(G')=A(G)$. On the other hand, we have $A(T)\subseteq A(S)^\theta$, $A(G')^\theta\subseteq Z(H')$ and $A(G)^\theta\subseteq Z(H)$. Hence $A(T)\subseteq Z(H)\subseteq Z(H')$. Thus, to prove the lemma, it suffices to show that
\begin{equation}\label{equ. determinant 1}
\det\left(\Ad(t)|_{\fh(F)_{x,s:s+}}\right)=\det\left(\Ad(t)|_{\fh'(F)_{x,s:s+}}\right)=1,\
\forall\ t\in T(F)_0.\end{equation} Let $\mathsf{T}$ be the special
fiber of the connected Neron model of $T$, which is a subtorus of
$\mathsf{H'}_x$ and $\mathsf{H}_x$. Then (\ref{equ. determinant 1})
is equivalent to
\begin{equation}\label{equ. determinant 2}
\det\left(\Ad(t)|_{\fh(F)_{x,s:s+}}\right)=\det\left(\Ad(t)|_{\fh'(F)_{x,s:s+}}\right)=1,\
\forall\ t\in \mathsf{T}(k_F).\end{equation} Denote
$\mathsf{V}_{x,s}=\fh(F)_{x,s:s+}$, which is viewed as a
$k_F$-affine space. The adjoint action of $\mathsf{H}_x$ on
$\mathsf{V}_{x,s}$ is an algebraic representation. Hence
$\det\left(\Ad(\cdot)|_{\mathsf{V}_{x,s}}\right)$ is an algebraic
character of $\mathsf{H}_x$. Since $\mathsf{H}_x=\mathsf{Z}(\sfH_x)^\circ\cdot(\sfH_x)_\der$  and
the restriction of this algebraic character to
$\mathsf{Z}(\mathsf{H}_x)^\circ$ and $(\sfH_x)_\der$ is trivial,
$\det\left(\Ad(\cdot)|_{\mathsf{V}_{x,s}}\right)$ itself is the
trivial character. By the same reason,
$\det\left(\Ad(\cdot)|_{\fh'(F)_{x,s:s+}}\right)$ is also trivial.
We conclude that (\ref{equ. determinant 2}) holds.
\end{proof}

\begin{remark}
As pointed out by a referee, our proof of Lemma \ref{lem. key lemma} is nearly identical to that of \cite[Lemma 7.10]{hak13}, except that the case $G=\GL_n$ is treated therein. We also remark that, according to the above proof, Lemma \ref{lem. key lemma} holds under the Hypothesis 2.1.1 of \cite{adl98} on $p$.
\end{remark}

\subsubsection{$(\theta,\varepsilon\eta)$-symmetric pair}\label{subsubsec. twisted theta
symmetric}

\begin{defn}\label{defn. twisted theta symmetric}
Let $(S,\mu)$ be a tame regular elliptic pair of $G$. We say that
$(S,\mu)$ is {\em $(\theta,\varepsilon\eta)$-symmetric} if:
\begin{itemize}
\item $S$ is $\theta$-stable,
\item $\mu|_{S^\theta(F)}=\varepsilon_S\cdot\eta_S$.
\end{itemize}
\end{defn}

\begin{lem}\label{lem. theta Howe factorization}
Let $(S,\mu)$ be a $(\theta,\varepsilon\eta)$-symmetric tame regular
elliptic pair. Then there exists a Howe factorization
$(\vec{G},\pi_{(S,\mu_\circ)},\vec{\phi})$ of $(S,\mu)$ such that
\begin{equation}\label{equ. theta Howe factorization 1}
\phi_i|_{(S^u)^{\theta,\circ}(F)}=1\end{equation} and
\begin{equation}\label{equ. theta Howe factorization 2}
\phi_i|_{G^0(F)^\theta_{x,0+}}=1\end{equation} for each $0\leq i\leq
d$.
\end{lem}

\begin{proof}
Denote $T=(S^u)^{\theta,\circ}$. According to Lemmas \ref{lem.
character varepsilon 2} and \ref{lem. key lemma}, we have
$\mu|_{T(F)}=1$. From the definition of $\varepsilon_S$ and
$\eta_S$, it is easy to see that $\mu|_{S(F)^\theta_{0+}}=1$. To
prove the assertion of this lemma, it suffices to plug the
conditions (\ref{equ. theta Howe factorization 1}) and (\ref{equ.
theta Howe factorization 2}) into the proof of \cite[Proposition
3.6.7]{kal}, which establishes the existence of Howe factorizations
by a recursive construction, and check that the same recursion goes
through in our situation. We just point out that the only necessary
modification that we need is a stronger statement of \cite[Lemma
3.6.9]{kal}: we further require that the character $\phi:
G(F)\ra\BC^\times$ satisfies $\phi|_{T(F)}=1$ and
$\phi|_{G^0(F)^\theta_{x,0+}}=1$ besides
$\phi|_{S(F)_r}=\mu|_{S(F)_r}$. Here we use the same notation as
that in the proof of \cite[Lemma 3.6.9]{kal}. Now let us prove this
statement. Let $M_1$ and $M_2$ be the images of $T(F)$ and
$G^0(F)^\theta_{x,0+}$ in $D(F)=(G/G_\der)(F)$ respectively. Let
$M_3$ be the subgroup of $D(F)$ generated by $M_1$ and
$D(F)_r=S(F)_r/S_\der(F)_r$, and $M_4$ the subgroup of $D(F)$
generated by $M_2$ and $M_3$. Then $\mu$ descends to a non-trivial
finite order character of $M_3$ which is trivial on $M_1$ and
$D(F)_{r+}$. Since $G^0(F)^\theta_{x,0+}\cap S(F)=S(F)^\theta_{0+}$,
this character of $M_3$ can be extended uniquely to a character
$\phi'$ of $M_4$ which is trivial on $M_2$. Then $\phi'$ can be
extended to a character $\phi$ of $D(F)$, whose pull-back to $G(F)$
satisfies our requirement.
\end{proof}

\begin{cor}\label{cor. theta symmetric}
Let $(S,\mu)$ be a $(\theta,\varepsilon\eta)$-symmetric tame regular
elliptic pair. Then any Howe factorization
$(\vec{G},\pi_{(S,\mu_\circ)},\vec{\phi})$ of $(S,\mu)$ satisfies
$$\phi|_{G^0(F)^\theta_{x,0+}}=1$$ where
$\phi=\prod\limits_{i=0}^d\phi_i$ is viewed as a character of
$G^0(F)$.
\end{cor}

\begin{proof}
It is a direct consequence of Lemma \ref{lem. theta Howe
factorization} and the fact that $\phi|_{G^0(F)_{x,0+}}$ are the
same for all the Howe factorizations since they differ by
refactorizations.
\end{proof}

\begin{remark}\label{rem. twsited theta symmetric depth-zero}
For the depth-zero case, i.e., when $S$ is a maximally unramified
elliptic maximal torus of $G$ and $\mu$ a regular depth-zero
character of $S(F)$, we abbreviate the notion
$(\theta,\varepsilon\eta)$-symmetric to be
$(\theta,\varepsilon)$-symmetric since $\eta_S$ is trivial in this
situation. In this case, if $(S,\mu)$ is
$(\theta,\varepsilon)$-symmetric, according to Lemma \ref{lem.
character varepsilon 2}, $(S,\mu)$ is $\theta$-symmetric, i.e., we
have $\mu^{-1}=\mu\circ\theta$. This is because
$S(F)^{1+\theta}\subseteq S^{\theta,\circ}(F)$.
\end{remark}

\begin{remark}\label{rem. theta-symmetric characters}
For positive depth case, according to Lemma \ref{lem. character
varepsilon 2}, the condition
$\mu|_{S^\theta(F)}=\varepsilon_S\cdot\eta_S$ implies that
$$\mu|_{S^{1+\theta}(F)}=\eta_S|_{S^{1+\theta}(F)}.$$
However, we could not show $\eta_S|_{S^{1+\theta}(F)}=1$, i.e.,
$\eta_S$ is $\theta$-symmetric. To the best of the author's knowledge,
$\eta_S$ is $\theta$-symmetric for all the examples studied by Hakim
and his collaborators. We speculate that it also holds in general.
\end{remark}

\subsubsection{Statement of the main theorem}\label{subsubsec. main theorems}

Now  we can state our main theorem on the distinction problem.
\begin{thm}\label{thm. main theorem}
Let $\pi_{(S,\mu)}$ be a regular supercuspidal representation of
$G(F)$. Then $\pi_{(S,\mu)}$ is $H$-distinguished if and only if
$(S,\mu)$ is $G(F)$-conjugate to a
$(\theta,\varepsilon\eta)$-symmetric tame regular elliptic pair.
\end{thm}

\begin{cor}\label{cor. cor of main theorem}
Let $\pi$ be a regular depth-zero supercuspidal representation of
$G(F)$. If $\pi$ is $H$-distinguished then we have
$\pi^\vee\simeq\pi\circ\theta$.
\end{cor}

\begin{proof}
For a regular supercuspidal representation $\pi_{(S,\mu)}$, it follows from \cite[Theorem 4.25, Corollary 4.26]{hm08} that \begin{equation}\label{equ. contragredient}\pi_{(S,\mu)}^\vee\simeq\pi_{(S,\mu^{-1})}\quad\textrm{and}\quad
\pi_{(S,\mu)}\circ\theta\simeq\pi_{(\theta(S),\mu\circ\theta)}.\end{equation} Now let $\pi$ be an $H$-distinguished regular depth-zero supercuspidal
representation. According to Theorem \ref{thm. main theorem} and
Remark \ref{rem. twsited theta symmetric depth-zero}, we can choose
a $(\theta,\varepsilon)$-symmetric maximally unramified regular
elliptic pair $(S,\mu)$ such that $\pi\simeq\pi_{(S,\mu)}$.
By Remark
\ref{rem. twsited theta symmetric depth-zero}, the condition that
$(S,\mu)$ is $\theta$-symmetric implies the corollary.
\end{proof}

\subsubsection{Some examples}\label{subsec. supplements}

\paragraph{\bf Galois involution.}
Let $H$ be a connected reductive group over $F$, $E$ a quadratic
field extension of $F$, and $G=\R_{E/F}H$ the Weil restriction of
$H$ with respect to $E/F$. The non-trivial automorphism of
$\Gal(E/F)$ gives rise to an involution $\theta$ of $G$, which is
called a {\em Galois involution}. If $E/F$ is unramified, we call $\theta$ an {\em unramified Galois involution}.

\begin{cor}\label{prop. galois involution} Let $\theta$ be a
Galois involution of $G$ and $\pi_{(\dot{S},\dot{\mu})}$
a regular supercuspidal representation of $G(F)$. Then
$\pi_{(\dot{S},\dot{\mu})}$ is $H$-distinguished if and only if
$(\dot{S},\dot{\mu})$ is $G(F)$-conjugate to a pair $({S},{\mu})$
such that ${S}$ is $\theta$-stable and
${\mu}|_{{S}^\theta(F)}=\eta_S$.
\end{cor}
\begin{proof}
 According to Theorem
\ref{thm. main theorem}, it suffices to show that $\varepsilon_S$ is
trivial if $S$ is $\theta$-stable. Since $S$ is $\theta$-stable, by
Galois descent, we have $S=\R_{E/F}T$ where
$T=S^{\theta}=S^{\theta,\circ}$ is a torus of $H$. Therefore, by
Lemma \ref{lem. character varepsilon 2}, $\varepsilon_S$ is trivial.
\end{proof}

\begin{remark}\label{rem. galois involution}
For Galois involutions, Prasad \cite{pra} stated a precise conjecture
to give sufficient and necessary conditions for representations to
be distinguished in terms of the Langlands parameters. Our prior work
\cite{zha} verified a necessary condition of this conjecture for
unramified Galois involutions, when the representations are regular
depth-zero supercuspidal representations of unramified groups. The
above corollary is a generalization of \cite[Proposition 3.2]{zha}.
Moreover, in a sequal \cite{zhab} to this paper, we prove that $\eta_S$ is indeed trivial for 
unramified Galois involutions.
\end{remark}

\paragraph{\bf Epipelagic supercuspidal
representations.} The notion of epipelagic supercuspidal representations was first
introduced by Reeder and Yu \cite{ry14}. Kaletha \cite{kal15} later
studied the properties of epipelagic $L$-packets, including the
endoscopic character identities. This kind of representations is a
special case of a more general class of supercuspidal
representations, called toral supercuspidal representations which
were first considered by Adler \cite{adl98}. We refer to \cite[\S2.5]{ry14}
for the definition of epipelagic supercuspidal representations, and to \cite[\S6]{kal}
for a brief discussion on toral supercuspidal representations. In
terms of Yu's data, epipelagic supercuspidal representations are
constructed from generic cuspidal {\em epipelagic} $G$-data
$$\left((G^0=S,G^1=G),x,\rho=1,(\phi_0=\mu,\phi_1=1)\right),$$ where
$x\in\CB^\red(G,F)$ is a rational point of order $e$ (cf.
\cite[\S3.3]{ry14}) and $(S,\mu)$  a tame regular elliptic pair
satisfying \cite[Conditions 3.3]{kal15}. The resulting
representations $\pi_{(S,\mu)}$ are called {\em epipelagic}. An
important property of epipelagic $G$-data is that
$$G(F)_{x,\frac{1}{2e}+}=G(F)_{x,\frac{1}{2e}}=G(F)_{x,\frac{1}{e}}.$$
Therefore $J^1/J^1_+$ is automatically trivial, and thus
$\eta_\theta=1$ for any epipelagic $G$-datum such that
$[\theta]\sim[\Psi]$. Then the following corollary is a direct
consequence of Theorem \ref{thm. HM distinction} or
\cite[Proposition 5.31]{hm08}.

\begin{cor}\label{cor. epipelagic}
Let $\pi_{(\dot{S},\dot{\mu})}$ be an epipelagic supercuspidal
representation. Then $\pi_{(\dot{S},\dot{\mu})}$ is
$H$-distinguished if and only if $(\dot{S},\dot{\mu})$ is
$G(F)$-conjugate to a pair $({S},{\mu})$ such that ${S}$ is
$\theta$-stable and ${\mu}|_{{S}^\theta(F)}=1$. In particular, if
$\pi_{(S,\mu)}$ is $H$-distinguished then we have
$\pi_{(S,\mu)}^\vee\simeq\pi_{(S,\mu)}\circ\theta$.

\end{cor}

\subsection{Proof of the main theorem}\label{subsec. proofs}

\subsubsection{Distinction over field fields}
Let $\mathsf{G}$ be a connected reductive group over $k_F$ and
$\mathsf{T}$ a maximal $k_F$-torus of $\mathsf{G}$. Let $\lambda$ be
a character of $\mathsf{T}(k_F)$ in general position and
$\kappa_{(\mathsf{T},\lambda)}=\pm R_{\mathsf{T},\lambda}$ the
Deligne--Lusztig representation of $\mathsf{G}(k_F)$. Let $\theta$
be an involution of $\mathsf{G}$ defined over $k_F$,
$\mathsf{H}=\mathsf{G}^\theta$, and $\eta$ a character of
$\mathsf{H}(k_F)$. Denote
$$\mathsf{m}=\dim\Hom_{\mathsf{H}(k_F)}\left(\kappa_{(\mathsf{T},\lambda)},\eta\right).$$
We call $(\mathsf{T},\lambda)$ a
{\em$(\theta,\varepsilon\eta)$-symmetric pair} if $\mathsf{T}$ is
$\theta$-stable and
$$\lambda|_{\mathsf{T}^\theta(k_F)}
=\varepsilon_\mathsf{T}\cdot\eta|_{\mathsf{T}^\theta(k_F)}.$$ Recall that the definition of $\varepsilon_{\mathsf{T}}$ is given at the beginning of \S\ref{subsubsec. character varepsilon}. The
following lemma is a partial summary of prior works \cite{lus90},
\cite[\S3.2]{hj12} and \cite[\S8.2]{hak13}.

\begin{lem}\label{lem. disctinction finite field }
If the multiplicity $\mathsf{m}$ is non-zero, then
$(\mathsf{T},\lambda)$ is $\mathsf{G}(k_F)$-conjugate to a
$(\theta,\varepsilon\eta)$-symmetric pair. If we further assume that
$\eta|_{\mathsf{T}_1^{\theta,\circ}(k_F)}=1$ for any $\theta$-stable
torus $\mathsf{T}_1$ that is $\mathsf{G}(k_F)$-conjugate to
$\mathsf{T}$, the converse also holds.
\end{lem}

\begin{remark}\label{rem. finite field}
When $\lambda$ is an arbitrary character and $\eta=1$, Lusztig
\cite[Theorem 3.3]{lus90} established an explicit formula for $\mathsf{m}$. Hakim
and Lansky \cite[Theorem 3.11]{hj12} generalized Lusztig's formula to
arbitrary $\eta$. When $\lambda$ is in general position, the multiplicity
formula for $\mathsf{m}$ becomes much more simple, as discussed in
\cite[\S10]{lus90} and \cite[\S8.2]{hak13}. The above lemma can be
deduced directly from the multiplicity formula for $\mathsf{m}$.
\end{remark}

\subsubsection{Distinction of $\tilde{\kappa}_{(S,\mu)}$}\label{subsubsec.
multiplicity}
Let $G$ be a tamely ramified connected reductive group over $F$ and
$\theta$ an involution of $G$. Let $S$ be a maximally unramified
elliptic maximal torus of $G$ and $S^u$ the maximal unramified
subtorus of $S$. Let $x$ be the vertex of $\CB^\red(G,F)$ attached
to $S$.  Let $\mu$ be a regular depth-zero character of $S(F)$ and
$\tilde{\kappa}_{(S,\mu)}$ the representation of
$G_S=S(F)G(F)_{x,0}$ introduced in Section \ref{subsubsec. tame
regular pairs depth-zero case}. Suppose that $\theta(x)=x$ and $G_S$
is $\theta$-stable. Then $G(F)_{x,0}$ and $G(F)_{x,0+}$ are both
$\theta$-stable. Let $\eta$ be a character of $G_S^\theta$ which is
trivial on $Z^\theta(F)$ and $G(F)^\theta_{x,0+}$. One of the key steps to prove
Theorem \ref{thm. main theorem} is to determine
when the multiplicity
$$\fm:=\dim\Hom_{G_S^\theta}\left(\tilde{\kappa}_{(S,\mu)},\eta\right)$$
is non-zero. In this subsection, we say that $(S,\mu)$ is
$(\theta,\varepsilon\eta)$-symmetric if $S$ is $\theta$-stable and
$$\mu|_{S^\theta(F)}=\varepsilon_{S}\cdot\eta|_{S^\theta(F)}.$$

\begin{prop}\label{prop. multiplicity}
If the multiplicity $\fm$ is non-zero, then $(S,\mu)$ is
$G(F)_{x,0}$-conjugate to a $(\theta,\varepsilon\eta)$-symmetric
pair. If we further assume that
$\eta|_{(S^u_1)^{\theta,\circ}(F)_0}=1$ for any $\theta$-stable
torus $S_1$ that is $G(F)_{x,0}$-conjugate to $S$, the converse also
holds.

\end{prop}

\begin{proof} For simplicity we will denote
$\tilde{\kappa}_{(S,\mu)}$ by $\tilde{\kappa}$ when there is no
confusion. First note that $Z(F)$ acts on the representation space
$V$ of $\tilde{\kappa}$ by the restriction of $\mu$ to $Z(F)$. Hence
a necessary condition for the non-vanishing of $\fm$ is
\begin{equation}\label{equ. central character}\mu|_{Z^\theta(F)}=1.\end{equation}
From now on we assume (\ref{equ. central character}). Denote
$$\mathsf{M}=G_S^\theta/\left(Z(F)^\theta
G(F)^\theta_{x,0+}\right).$$ Since $G(F)_{x,0+}$ acts trivially on
$V$, we have
$$\fm=\dim\Hom_\mathsf{M}(\tilde{\kappa},\eta).$$
We claim that $\mathsf{M}$ is a finite group. Note that
$$\begin{aligned}G_S^\theta\cap
\left(Z(F)G(F)_{x,0+}\right)=&\left(Z(F)G(F)_{x,0+}\right)^\theta\\
=&Z(F)^\theta G(F)^\theta_{x,0+},\end{aligned}$$ where the last
equality is due to \cite[Lemma 2.11, Proposition 2.12]{hm08}.
Therefore $\mathsf{M}$ is a subgroup of
$G_S/\left(Z(F)G(F)_{x,0+}\right)$ and the latter group is obviously
a finite group since $G_S/Z(F)$ is compact.

Since the group $\mathsf{M}$ is finite and the space $V$ is finite
dimensional, we have
$$\fm=\dim V^\mathsf{(M,\eta)}
=\frac{1}{\abs{\mathsf{M}}}\sum_{\gamma\in\mathsf{M}}\Theta(\gamma)\eta^{-1}(\gamma),$$
where $V^{(\mathsf{M},\eta)}$ is the isotypical subspace of $V$ on
which $\mathsf{M}$ acts by $\eta$, and $\Theta$ is the character of
the representation $\tilde{\kappa}$.

Now we review the character formula of $\Theta$ \cite[Propositions
3.4.23 and 3.4.24]{kal}. The notation below is the same as that in Section
\ref{subsubsec. character varepsilon}. We view $\tilde{\kappa}$ as a
representation of $\mathsf{G}_S(k_F)=G_S/G(F)_{x,0+}$. For
$\gamma=rg\in\mathsf{G}_S(k_F)$ with
$r\in\mathsf{S}(k_F)=S(F)/S(F)_{0,+}$ and
$g\in\mathsf{G}_x(k_F)=G(F)_{x,0}/G(F)_{x,0+}$, there exists a
Jordan decomposition $\gamma=\gamma_s\gamma_u$ given as follows; see the paragraph below {\em loc. cit.} Proposition 3.4.23 and the first two paragraphs of the proof of {\em loc. cit.} Propositions 3.4.23 and 3.4.24. Let
$\bar{r}$ be the image of $r$ in $\mathsf{S^u_\ad}(k_F)$ and
$\dot{r}$ any lift of $\bar{r}$ in $\mathsf{S^u}(\bar{k}_F)$. Let
$\dot{r}g=su$ be the Jordan decomposition of $\dot{r}g$ in
$\mathsf{G}_x(\bar{k}_F)$. In fact we have
$\dot{r}^{-1}s\in\sfG_x(k_F)$ and
$u\in\mathsf{G}_x(k_F)_\unip$ where $\mathsf{G}_x(k_F)_\unip$
denotes the set of unipotent elements of $\mathsf{G}_x(k_F)$. Set
$\gamma_s=r\dot{r}^{-1}s\in\sfG_S(k_F)$ and $\gamma_u=u\in\sfG_x(k_F)$. This
decomposition is independent of the choice of $\dot{r}$ and thus is
unique. Moreover $r\dot{r}^{-1}$ commutes with any element of
$\mathsf{G}_x(\bar{k}_F)$ and
$C_{\mathsf{G}_x}(\gamma_s)=C_{\mathsf{G}_x}(s)$ is defined over
$k_F$. Then the character formula is
\begin{equation}\label{equ. character formula}
\Theta(\gamma)={\sigma(\mathsf{G}_x)\sigma(\mathsf{S^u})}
\frac{1}{\abs{C_{\mathsf{G}_x}(\gamma_s)(k_F)}}
\sum_{\mbox{\tiny$\begin{array}{c}y\in\mathsf{G}_x(k_F)\\
y^{-1}\gamma_s y\in\mathsf{{S}}(k_F)\end{array}$}}
\mu(y^{-1}\gamma_s
y)Q^{C_{\mathsf{G}_x}(\gamma_s)}_{y\mathsf{S^u}y^{-1},1}(\gamma_u),
\end{equation}
where
$Q^{C_{\mathsf{G}_x}(\gamma_s)}_{y\mathsf{S^u}y^{-1},1}(\gamma_u)$ is
the Green function. 

We remark that the character formulas in {\em loc. cit.} Propositions 3.4.23 and 3.4.24 are valid for elements in $G_S$, and the formula (\ref{equ. character formula}) is valid for elements in $\sfG_S(k_F)$. The proof of (\ref{equ. character formula}) can be read off from that of {\em loc. cit.} Propositions 3.4.23 and 3.4.24.

From now on, for convenience, we denote
$\mathsf{G}=\mathsf{G}_x$.
Passage from $\mathsf{G}_S(k_F)$ to $\mathsf{M}$, for
$\gamma\in\mathsf{M}$ we have Jordan decomposition
$\gamma=\gamma_s\gamma_u$ with
$\gamma_s\in\mathsf{G}_S(k_F)/Z(F)^\theta$ and
$\gamma_u\in\mathsf{G}(k_F)_\unip$. Since $\theta(\gamma)=\gamma$,
by the uniqueness of Jordan decomposition, it has to be
$\gamma_s\in\mathsf{M}$ and
$\gamma_u\in\mathsf{G}(k_F)^\theta_\unip$. Put
$$\bar{\mathsf{S}}(k_F)=S(F)/Z(F)S(F)_{0+}.$$ We denote by
$\mathsf{M}_\ss$ the semisimple part of $\mathsf{M}$. Set
$$\chi=\eta^{-1}.$$
The following computation of $\fm$ is a modification of that in the
proof of \cite[Theorem 3.11]{hj12} which is based on the proof of
the main result of \cite[Theorem 3.3]{lus90}. First, by the Jordan decomposition, we
have
$$\begin{aligned}
\fm
&=\frac{1}{\abs{\mathsf{M}}}\sum_{\gamma_s\gamma_u\in\mathsf{M}}\Theta(\gamma_s\gamma_u)\chi(\gamma_s\gamma_u)\\
&=\frac{1}{\abs{\mathsf{M}}}\sum_{\gamma_s\in\mathsf{M}_\ss}\chi(\gamma_s)\sum_{\gamma_u\in
C_\mathsf{G}(\gamma_s)(k_F)\cap\mathsf{G}(k_F)^\theta_\unip}\Theta(\gamma_s\gamma_u)\\
&=\frac{{\sigma(\mathsf{G})\sigma(\mathsf{S^u})}}{\abs{\mathsf{M}}}\sum_{\gamma_s\in
\mathsf{M}_\ss}\chi(\gamma_s)\sum_{\gamma_u\in
C_\mathsf{G}(\gamma_s)(k_F)\cap\mathsf{G}(k_F)^\theta_\unip}\frac{1}{\abs{C_\mathsf{G}(\gamma_s)(k_F)}}\\
&\quad \cdot\sum_{\mbox{\tiny$\begin{array}{c}y\in\mathsf{G}(k_F)\\
y^{-1}\gamma_s y\in\bar{\mathsf{S}}(k_F)\end{array}$}}
\mu(y^{-1}\gamma_s y)Q^{C_\mathsf{G}(\gamma_s)}_{y\mathsf{S^u}y^{-1},1}(\gamma_u)\\
&=\frac{{\sigma(\mathsf{G})\sigma(\mathsf{S^u})}}{\abs{\mathsf{M}}}\sum_{\gamma_s\in
\mathsf{M}_\ss}
\sum_{\mbox{\tiny$\begin{array}{c}y\in\mathsf{G}(k_F)\\
y^{-1}\gamma_sy\in\bar{\mathsf{S}}(k_F)\end{array}$}}
\frac{\mu(y^{-1}\gamma_sy)\chi(\gamma_s)}{\abs{C_\mathsf{G}(\gamma_s)(k_F)}}\\
&\quad\cdot\sum_{\gamma_u\in
C_\mathsf{G}(\gamma_s)(k_F)\cap\mathsf{G}(k_F)^\theta_\unip}Q^{C_\mathsf{G}(\gamma_s)}_{y\mathsf{S^u}y^{-1},1}(\gamma_u).
\end{aligned}$$
By \cite[Theorem 3.4]{lus90}, we have
$$\begin{aligned}&\sum_{\gamma_u\in
C_\mathsf{G}(\gamma_s)(k_F)\cap\mathsf{G}(k_F)^\theta_\unip}
Q^{C_\mathsf{G}(\gamma_s)}_{y\mathsf{S^u}y^{-1},1}(\gamma_u)\\
=&\frac{\sigma(\mathsf{S^u})}{\abs{\mathsf{S^u}(k_F)}}\sum_{\mbox{\tiny$\begin{array}{c}g\in C_\mathsf{G}(\gamma_s)(k_F)\\
(y^{-1}g\cdot\theta)(\mathsf{S^u})=\mathsf{S^u}\end{array}$}}
\sigma\left(C_{C_\mathsf{G}(\gamma_s)}\left(({^{y^{-1}g}\mathsf{S^u}})^{\theta,\circ}\right)\right).
\end{aligned}$$
Therefore
$$\begin{aligned}\fm&=\frac{\sigma(\mathsf{G})}
{\abs{\mathsf{M}}\cdot\abs{\mathsf{S^u}(k_F)}}\sum_{\gamma\in
\mathsf{M}_\ss}
\sum_{\mbox{\tiny$\begin{array}{c}y\in\mathsf{G}(k_F)\\
y^{-1}\gamma y\in\mathsf{\bar{S}}(k_F)\end{array}$}}
\frac{\mu(y^{-1}\gamma y)\chi(\gamma)}{\abs{C_\mathsf{G}(\gamma)(k_F)}}\\
&\quad\cdot
\sum_{\mbox{\tiny$\begin{array}{c}g\in C_\mathsf{G}(\gamma)(k_F)\\
(y^{-1}g\cdot\theta)(\mathsf{S^u})=\mathsf{S^u}\end{array}$}}
\sigma\left(C_{C_\mathsf{G}(\gamma)}\left(({^{y^{-1}g}\mathsf{S^u}})^{\theta,\circ}\right)\right)
\end{aligned}$$
Changing variables $y^{-1}\gamma y\mapsto\gamma_1$ and
$y^{-1}g\mapsto y_1$, we obtain
$$\begin{aligned}
\fm=\frac{\sigma(\mathsf{G})}
{\abs{\mathsf{M}}\cdot\abs{\mathsf{S^u}(k_F)}}
\sum_{\mbox{\tiny$\begin{array}{c}(\gamma_1,y_1)\in
\bar{\mathsf{S}}(k_F)\times\mathsf{G}(k_F)\\
y^{-1}_1\gamma_1y_1\in\mathsf{M}\\
(y_1\cdot\theta)(\mathsf{S^u})=\mathsf{S^u}\end{array}$}}\mu(\gamma_1)\chi(y_1^{-1}\gamma_1y_1)
\sigma\left(C_{C_\mathsf{G}(y_1^{-1}\gamma_1y_1)}
\left(({^{y_1}\mathsf{S^u}})^{\theta,\circ}\right)\right).
\end{aligned}$$
For each $y_1\in\mathsf{G}(k_F)$ in the above summation, we choose
an arbitrary lift $\dot{y}_1\in G(F)_{x,0}$ of $y_1$. The maximal
torus of $\mathsf{G}$ which corresponds to $^{\dot{y}_1}S$ is
$^{y_1}\mathsf{S^u}$. Hence by Lemma \ref{lem. stable torus} there
exists a $\theta$-stable torus $S_1$ which is
$G(F)_{x,0+}$-conjugate to $^{\dot{y}_1}S$ and thus
$G(F)_{x,0}$-conjugate to $S$. We have
$\mathsf{S}_1(k_F)=y_1^{-1}\mathsf{S}(k_F)y_1$ which is
$\theta$-stable. Denote by $\varepsilon_{^{y_1}\mathsf{S}}$ the
character $\varepsilon_{\mathsf{S_1}}$ defined by (\ref{equ.
chracter varepsilon}). Note that $\varepsilon_{^{y_1}\mathsf{S}}$ is
well defined since it is independent of the choices of $\dot{y}_1$
and $S_1$. According to the definition of the character
$\varepsilon_{^{y_1}\mathsf{S}}$, we have
$$\begin{aligned}
\sigma\left(C_{C_\mathsf{G}(y_1^{-1}\gamma_1y_1)}
\left(({^{y_1}\mathsf{S^u}})^{\theta,\circ}\right)\right)
&=\sigma\left(C_\mathsf{G}
\left(({^{y_1}\mathsf{S^u}})^{\theta,\circ}\right)\cap
C_\mathsf{G}\left(y_1^{-1}\gamma_1y_1\right)\right)\\
&=\varepsilon_{^{y_1}\mathsf{S}}(y_1^{-1}\gamma_1y_1)
\sigma\left(C_\mathsf{G}
\left(({^{y_1}\mathsf{S^u}})^{\theta,\circ}\right)\right).
\end{aligned}$$
Thus,
$$\fm=\frac{\sigma(\mathsf{G})}
{\abs{\mathsf{M}}\cdot\abs{\mathsf{S^u}(k_F)}}
\sum_{\mbox{\tiny$\begin{array}{c}(\gamma_1,y_1)\in
\bar{\mathsf{S}}(k_F)\times\mathsf{G}(k_F)\\
y^{-1}_1\gamma_1y_1\in\mathsf{M}\\
(y_1\cdot\theta)(\mathsf{S^u})=\mathsf{S^u}\end{array}$}}\mu(\gamma_1)\chi(y_1^{-1}\gamma_1y_1)
\varepsilon_{^{y_1}\mathsf{S}}(y_1^{-1}\gamma_1y_1)
\sigma\left(C_\mathsf{G}
\left(({^{y_1}\mathsf{S^u}})^{\theta,\circ}\right)\right).
$$
Changing variables $y_1^{-1}\gamma_1y_1\mapsto\gamma_2$, we get
$$\begin{aligned}
\fm=&\frac{\sigma(\mathsf{G})}
{\abs{\mathsf{M}}\cdot\abs{\mathsf{S^u}(k_F)}}
\sum_{\mbox{\tiny$\begin{array}{c}y_1\in\mathsf{G}(k_F)\\
(y_1\cdot\theta)(\mathsf{S^u})=\mathsf{S^u}\end{array}$}}
\sigma\left(C_\mathsf{G}
\left(({^{y_1}\mathsf{S^u}})^{\theta,\circ}\right)\right)\\
&\cdot\sum_{\gamma_2\in \left(y_1^{-1}\bar{\mathsf{S}}(k_F)y_1\right)\cap \mathsf{M}}
{(^{y_1}\mu)}(\gamma_2)\chi(\gamma_2)
\varepsilon_{^{y_1}\mathsf{S}}(\gamma_2).
\end{aligned}$$
The term
$$\fm_{y_1}:=\sum_{\gamma_2\in \left(y_1^{-1}\bar{\mathsf{S}}(k_F)y_1\right)\cap \mathsf{M}}
{(^{y_1}\mu)}(\gamma_2)\chi(\gamma_2)
\varepsilon_{^{y_1}\mathsf{S}}(\gamma_2)$$ is a positive integer
precisely when
\begin{equation}\label{equ. character relation}
^{y_1}\mu|_{\left(^{y_1}\mathsf{S}\right)^\theta(k_F)}
=\varepsilon_{^{y_1}\mathsf{S}}\cdot\eta|_{\left(^{y_1}\mathsf{S}\right)^\theta(k_F)},\end{equation}
which is equivalent to
\begin{equation}\label{equ. character relation 2}
\mu_1|_{S_1^\theta(F)}
=\varepsilon_{S_1}\cdot\eta|_{S_1^\theta(F)},\end{equation} where
$\mu_1={^g\mu}$ and $g\in G(F)_{x,0}$ is such that $S_1={^gS}$.
Otherwise $\fm_{y_1}$ is zero. At this moment, we have proved the
first assertion of the proposition.

To prove the second assertion, first note that the relation
(\ref{equ. character relation}) implies that
\begin{equation}\label{equ. character relation 3}
\mu_1|_{(\mathsf{S^u_1})^\theta(k_F)}=\varepsilon_{\mathsf{S^u_1}}
\cdot\eta|_{(\mathsf{S^u_1})^\theta(k_F)}.\end{equation} Therefore,
according to (\ref{equ. character varepsilon relation}) and the
condition of the second assertion, (\ref{equ. character relation 3})
implies that
\begin{equation}\label{equ. character condition finite torus}
\mu_1|_{(\mathsf{S^u_1})^{\theta,\circ}(k_F)}=1.\end{equation} By
\cite[Lemmas 10.4 and 10.5]{lus90} and its slight generalization
\cite[Lemma 8.1]{hak13}, the condition (\ref{equ. character
condition finite torus}) implies that
$$\sigma\left(C_\mathsf{G}
\left(({\mathsf{S^u_1}})^{\theta,\circ}\right)\right)=\sigma(\mathsf{G}).$$
Therefore the multiplicity $\fm$ is equal to
$$\frac{1}
{\abs{\mathsf{M}}\cdot\abs{\mathsf{S^u}(k_F)}}
\sum_{\mbox{\tiny$\begin{array}{c}y\in\mathsf{G}(k_F)\\
(y\cdot\theta)(\mathsf{S^u})=\mathsf{S^u}\end{array}$}}\fm_y,$$
which implies the second assertion of the proposition directly.
\end{proof}

\subsubsection{Proof of Theorem \ref{thm. main theorem}}\label{subsubsec. proof general}

Now let $\pi=\pi_{(S,\mu)}$ be a regular supercuspidal
representation of $G(F)$ and
$\Psi=(\vec{G},\pi_{(S,\mu_\circ)},\vec{\phi})$ a Howe factorization
of $(S,\mu)$. Before proving Theorem \ref{thm. main theorem}, let us
remind the reader of the following notation that will be frequently
used:
\begin{itemize}
\item $\phi=\prod\limits_{i=0}^d \phi_i$, a character of $G^0(F)$,
\item $x\in\CB^\red(G^0,F)$ is the vertex determined by $S$,
\item $K^0=G^0(F)_x$ and $G^0_S=S(F)G^0(F)_{x,0}$,
\item $\kappa:=\kappa_{(S,\mu_\circ)}$ and $\tilde{\kappa}:=\tilde{\kappa}_{(S,\mu_\circ)}$
are representations of $G^0(F)_{x,0}$ and $G^0_S$ respectively,
which are constructed from the depth-zero tame regular elliptic pair $(S,\mu_\circ)$ for $G^0$, as described in \S\ref{subsubsec. tame regular pairs depth-zero case}, except that $G$ is replaced by $G^0$,
\item
$\rho:=\rho_{(S,\mu_\circ)}=\ind^{K^0}_{G^0_S}\tilde{\kappa}_{(S,\mu_\circ)}$,
\item we will abuse the notation to also denote by $\Psi$ the generic cuspidal $G$-datum $(\vec{G},x,\rho,\vec{\phi})$,
\item $\rho_\nm:=\rho\otimes\left(\phi|_{K^0}\right)$,
\item
$\tilde{\eta}_\theta:=\eta_\theta\cdot\left(\phi^{-1}|_{K^{0,\theta}}\right)$ if $[\theta]\sim[\Psi]$.

\end{itemize}

\paragraph{\bf Sufficient condition.}
Let us first prove the sufficient condition of Theorem \ref{thm.
main theorem}. According to Lemma \ref{lem. parameterize regular
repn} we can and do assume that the tame regular elliptic pair
$(S,\mu)$ is $(\theta,\varepsilon\eta)$-symmetric, which implies
that $[\theta]\sim[\Psi]$ by Corollary \ref{cor. theta symmetric}.
By Lemma \ref{lem. theta Howe factorization} we can further assume
that $\phi|_{(S^u)^{\theta,\circ}(F)}=1$. To show that $\pi$ is
$H$-distinguished, according to Theorem \ref{thm. HM distinction},
it suffices to show
$$\Hom_{K^{0,\theta}}\left(\rho_\nm,\eta_\theta\right)\neq0,$$ which
is equivalent to
$$\Hom_{K^{0,\theta}}\left(\rho,\tilde{\eta}_\theta\right)\neq0.$$
Since $\rho=\ind_{G_S^0}^{K^0}\tilde{\kappa}$, applying Mackey
theory, we have
\begin{equation}\label{equ. Mackey }\Hom_{K^{0,\theta}}(\rho,\tilde{\eta}_\theta)
\cong\bigoplus_{g\in G^0_S\bs K^0/K^{0,\theta}} \Hom_{G^0_S\cap
{^gK^{0,\theta}}}(\tilde{\kappa},{^g\tilde{\eta}_\theta}).\end{equation}
Consider the case when $g=1$. Note that $G^0_S\cap
K^{0,\theta}=G_S^{0,\theta}$. The condition that $(S,\mu)$ is
$(\theta,\varepsilon\eta)$-symmetric implies that $(S,\mu_\circ)$ is
$(\theta,\varepsilon\cdot\tilde{\eta}_\theta)$-symmetric in the
sense of Section \ref{subsubsec. multiplicity}. By Proposition
\ref{prop. multiplicity}, Lemma \ref{lem. key lemma} and the
condition $\phi|_{(S^u)^{\theta,\circ}(F)}=1$, we obtain that
$$\Hom_{G_S^{0,\theta}}(\tilde{\kappa},\tilde{\eta}_\theta)\neq0.$$ Thus
$\Hom_{K^{0,\theta}}(\rho_\nm,\eta_\theta)$ is non-zero.

\paragraph{\bf Necessary condition.}
Now let us turn to proving the necessary condition of Theorem
\ref{thm. main theorem}. According to Theorem \ref{thm. HM
distinction}, Theorem \ref{thm. HM G-equivalent} and Lemma \ref{lem.
parameterize regular repn}, we can and do assume that the generic cuspidal cuspidal $G$-datum $\Psi$ associated to
a Howe factorization of $(S,\mu)$ satisfies $[\theta]\sim[\Psi]$ and
$$\Hom_{K^{0,\theta}}(\rho_\nm,\eta_\theta)
=\Hom_{K^{0,\theta}}(\rho,\tilde{\eta}_\theta)\neq0.$$ Due to the
isomorphism (\ref{equ. Mackey }), there exists $k\in K^0$ such that
\begin{equation}\label{equ. main proof 1}\Hom_{G^0_S\cap
{^kK^{0,\theta}}}\left(\tilde{\kappa},{^{k}\tilde{\eta}_{\theta}}\right)\neq
0.\end{equation} Set $\theta'=k^{-1}\cdot\theta$. Then
$$^kK^{0,\theta}=K^{0,\theta'}\quad \textrm{and}\quad
{^k\tilde{\eta}_{\theta}}=\tilde{\eta}_{\theta'}.$$ Since $k\in
K^0$, we have $[\theta']=[\theta]\sim[\Psi]$, which implies that
$G^0(F)_{x,0}$ is $\theta'$-stable. Therefore,
$$G^0(F)_{x,0}^{\theta'}\subseteq G^0_S\cap K^{0,\theta'}.$$ Recall that
$\tilde{\kappa}$ is an extension of $\kappa$. Hence, by (\ref{equ.
main proof 1}), we have
\begin{equation}\label{equ. main proof 2}\Hom_{G^0(F)_{x,0}^{\theta'}}(\kappa,\tilde{\eta}_{\theta'})\neq 0.\end{equation} By \cite[Proposition 2.12]{hm08}, we see $G^0(F)^{\theta'}_{x,0}/G^0(F)^{\theta'}_{x,0+}=\mathsf{G^0}_x(k_F)^{\theta'}$. Hence (\ref{equ. main proof 2})
is equivalent to
$$\Hom_{\mathsf{G^0}_x(k_F)^{\theta'}}(\kappa,\tilde{\eta}_{\theta'})\neq 0.$$
Recall that $\kappa=\pm R_{\mathsf{S^u},\bar{\mu}_\circ}$. By Lemma
\ref{lem. disctinction finite field } we know that there exists
$\bar{y}\in\mathsf{G^0}_x(k_F)$ such that $^{\bar{y}}\mathsf{S^u}$
is $\theta'$-stable. Thus, by Lemma \ref{lem. stable torus}, there
exists $y\in G^0(F)_{x,0}$ such that ${^{y}S}$ is $\theta'$-stable.
Note that $G^0_S$ is also equal to $^yS(F)G^0(F)_{x,0}$, which
implies that $G^0_S$ is $\theta'$-stable. We deduce from (\ref{equ.
main proof 1}) that
$$\Hom_{G_S^{0,\theta'}}(\tilde{\kappa},\tilde{\eta}_{\theta'})\neq 0.$$
According to Proposition \ref{prop. multiplicity}, there exists
$z\in G^0(F)_{x,0}$ such that $^zS$ is $\theta'$-stable and
$$\begin{aligned}^z\mu_\circ|_{^zS^{\theta'}(F)}
&=\varepsilon_{^zS}\cdot\tilde{\eta}_{\theta'}|_{^zS^{\theta'}(F)}\\
&=\varepsilon_{^zS}\cdot\eta_{\theta'}\cdot\phi^{-1}|_{^zS^{\theta'}(F)},\end{aligned}$$
where the character $\varepsilon_{^zS}$ is defined with respect to
the involution $\theta'$. Therefore we have
$$^z\mu|_{^zS^{\theta'}(F)}=\varepsilon_{^zS}\cdot\eta_{\theta'}|_{^zS^{\theta'}(F)}.$$
Recall that $\theta'=k'\cdot\theta$ with $k'=k^{-1}\in K^0$. Since
$^zS$ is $\theta'$-stable, we have that $^{zk'}S$ is $\theta$-stable
and
$$^{zk'}\mu|_{^{zk'}S^\theta(F)}={^{k'}\varepsilon_{^zS}}\cdot
{^{k'}\eta_{\theta'}}|_{^{zk'}S^\theta(F)}.$$ It is easy to see that
$${^{k'}\eta_{\theta'}}=\eta_\theta,$$ and
$${^{k'}\varepsilon_{^zS}}=\varepsilon_{^{zk'}S}$$ where the
latter character is defined with respect to the involution $\theta$.
In summary we conclude that $({^{zk'}S},{^{zk'}\mu})$ is
$(\theta,\varepsilon\eta)$-symmetric.

\section{Functoriality}

In this section, we assume that $F$ has characteristic zero. Let $G$ be a connected tamely ramified quasi-split reductive group
over $F$, and $\wh{G}$
the complex Langlands dual group of $G$. We fix $\Gamma$-invariant splittings $(T,B,\{X_\alpha\})$ of $G$ and
$(\wh{T},\wh{B},\{X_{\wh{\alpha}}\})$ of $\wh{G}$. Recall that a splitting of a reductive group $G$, also called a pinning, is a triple $(T,B,\{X_\alpha\})$ where $T$ is maximal torus of $G$, $B\supset T$ is a Borel subgroup, and $X_\alpha$ is a non-zero vector in the root subspace $\fg_\alpha$ where $\alpha$ runs over the set of simple $B$-positive roots. Let $^LG=\wh{G}\rtimes
W_F$ be the Weil-form $L$-group.

\subsection{Regular supercuspidal $L$-packets}\label{subsec. regular supercuspidal
L-packets} In this subsection, we recall Kaletha's construction
\cite[\S5]{kal} of the compound $L$-packets $\Pi_\varphi$ for
regular supercuspidal $L$-parameters $\varphi$.

\subsubsection{Regular supercuspidal $L$-parameters and $L$-packet data}
The following Definitions \ref{defn. regular supercuspidal parameter} and \ref{defn. regular supercuspidal L-packet data} were introduced by Kaletha. See \cite[Lemma 5.2.2, Definitions 5.2.3 and 5.2.4]{kal} and also the paragraph of {\em loc. cit.} below Definition 5.2.3.
\label{subsubsec. regular sc parameter and datum}
\begin{defn}\label{defn. regular supercuspidal parameter}
We call a discrete $L$-parameter $\varphi:W_F\ra{^LG}$ {\em regular
supercuspidal} if it satisfies:
\begin{enumerate}
\item $\varphi(P_F)$ is contained in a torus of $\wh{G}$.
\item $C:=C_{\wh{G}}(\varphi(I_F))$ is a torus.
\item If $n\in N(\wh{T},\wh{M})$ projects onto a non-trivial element
of $\Omega(\wh{S},\wh{M})^\Gamma$, then $n$ does not belong to the
centralizer of $\varphi(I_F)$ in $\wh{G}$. Here we set
$\wh{M}$ to be $C_{\wh{G}}(\varphi(P_F))$, $\wh{T}$ to be $C_{\wh{M}}(C)$, and
$\wh{S}$ to be the $\Gamma$-module with underlying abelian group
$\wh{T}$ and the $\Gamma$-action given by $\Ad(\varphi(-))$.
\end{enumerate}
\end{defn}

\begin{defn}\label{defn. regular supercuspidal L-packet data}
We call a 4-tuple $(S,\wh{j},\chi,\mu)$ a {\em regular supercuspidal
$L$-packet datum} if it satisfies:
\begin{enumerate}
\item $S$ is a torus over $F$ of dimension equal to the absolute rank
of $G$ and splits over a tame extension of $F$,
\item $\wh{j}:\wh{S}\ra\wh{G}$ is an embedding of complex reductive
groups, whose $\wh{G}$-conjugacy class is $\Gamma$-stable. Then
$\wh{j}$ gives rise to a $\Gamma$-stable $G$-conjugacy class $J$ of
{\em admissible embeddings} $S\ra G$. Choose a $\Gamma$-fixed
element $j\in J$, which is defined over $F$, and identify $S$ with
its image $j(S)$ in $G$. We require that $S/Z(G)$ is anisotropic,
which means that $S$ is a tame elliptic maximal torus of $G$,
\item $\mu$ is a character of $S(F)$ such that $(S,\mu)$ is a tame
{\em extra} regular elliptic pair for $G$. The character $\mu$
determines a tamely ramified twisted Levi subgroup $G^0$ of $G$ and
a subgroup $\Omega(S,G^0)$ of $\Omega(S,G)$,
\item $\chi$ is $\Omega(S,G^0)(F)$-invariant {\em minimally ramified} $\chi$-data for
$R(S,G)$.
\end{enumerate}
\end{defn}

We have to explain the terminology in Definition \ref{defn. regular
supercuspidal L-packet data}. The notion admissible embeddings is
standard and is reviewed in \cite[\S5.1]{kal}. For a tame regular
elliptic pair $(S,\mu)$ of $G$, it is called a tame extra regular elliptic pair if the stabilizer
of $\mu|_{S(F)_0}$ in $\Omega(S,G^0)(F)$ is trivial. A set of
$\chi$-data is called minimally ramified if $\chi_\alpha=1$ for
asymmetric $\alpha$, $\chi_\alpha$ is unramified for unramified
symmetric $\alpha$, and $\chi_\alpha$ is tamely ramified for
ramified symmetric $\alpha$.

\begin{remark}\label{rem. category of L-packet data}
We can define morphisms, which are indeed isomorphisms, between
regular supercuspidal $L$-packet data. This enables us to view the
set of regular supercuspidal $L$-packet data as a category. See  {\em loc. cit.} Definition 5.2.5 for more details.
\end{remark}

The relation between regular supercuspidal $L$-packet data and regular supercuspidal parameter is discussed in \cite[Proposition 5.2.7]{kal}. First,
given a regular supercuspidal $L$-packet datum
$(S,\wh{j},\chi,\mu)$, let $$\varphi_{S,\mu}:W_F\ra{^LS}$$ be the
Langlands parameter corresponding to the character $\mu$ of $S(F)$,
and
$${^Lj_\chi}:{^LS}\ra{^LG}$$ the $L$-embedding extending $\wh{j}$
that is determined by the $\chi$-data $\chi$. Set
$\varphi={^Lj_\chi}\circ\varphi_{S,\mu}$. It is shown that $\varphi$ is a
regular supercuspidal parameter. Conversely, given a regular
supercuspidal parameter $\varphi$, Kaletha proved that there exists a regular
supercuspidal $L$-packet datum $(S,\wh{j},\chi,\mu)$ such that
$\varphi={^Lj_\chi}\circ\varphi_{S,\mu}$. The following is
{\em loc. cit.} Proposition 5.2.7.

\begin{prop}
\label{prop. correspondence between parameters and data} The above
process provides an 1-1 correspondence between the isomorphism
classes of regular supercuspidal $L$-packet data and the
$\wh{G}$-conjugacy classes of regular supercuspidal parameters.
\end{prop}

\subsubsection{Rigid inner twists}
Kaletha \cite[\S3.2]{kal16a} defined a set $Z^1(u\ra W,Z\ra G)$ of
cocycles and a cohomology set $H^1(u\ra W,Z\ra G)$ for any finite
central subgroup $Z$ of $G$, where $u$ is a multiplicative
pro-algebraic group and $W$ a fixed extension of $\Gamma$ by $u$. We
have natural maps $$Z^1(u\ra W,Z\ra G)\ra Z^1(\Gamma, G/Z)\ra
Z^1(\Gamma, G_\ad)$$ and
$$H^1(u\ra W,Z\ra G)\ra H^1(\Gamma, G/Z)\ra
H^1(\Gamma, G_\ad),$$ which are induced by the projection $G\ra
G/Z$. Recall that an inner twist $G\ra G'$ of $G$ gives rise to a
cocycle $z\in Z^1(\Gamma,G_\ad)$, and the set of isomorphism classes
of inner twists is parameterized by $H^1(\Gamma,G_\ad)$. We call the
triple $(G',\xi,z)$ a {\em rigid inner twist} of $G$ if $\xi:G\ra
G'$ is an inner twist and $z$ is a cocycle in  $Z^1(u\ra W,Z\ra G)$
for some $Z$ such that $\xi$ corresponds to the image of $z$ in
$Z^1(\Gamma, G_\ad)$. The set of isomorphism classes of rigid inner
twists given by $Z^1(u\ra W,Z\ra G)$ is parameterized by $H^1(u\ra
W,Z\ra G)$.

\subsubsection{Regular supercuspidal data and $L$-packets}

\begin{defn}\label{defn. regular supercuspidal data}
We call a tuple $(S,\wh{j},\chi,\mu,(G',\xi,z),j)$ a {\em regular
supercuspidal datum} if it satisfies:
\begin{enumerate}
\item $(S,\wh{j},\chi,\mu)$ is a regular supercuspidal $L$-packet
datum,
\item $(G',\xi,z)$ is a rigid inner twist of $G$,
\item $j:S\ra G'$ is an admissible embedding over $F$ with respect
to $\wh{j}$.
\end{enumerate}
\end{defn}

\begin{remark}\label{rem. category of data}
The above definition is \cite[Definition 5.3.2]{kal}.
We can also make the set of regular supercuspidal data being a
category. There is a natural forgetful functor from it onto the
category of regular supercuspidal $L$-packet data. Given a regular
supercuspidal $L$-packet datum $(S,\wh{j},\chi,\mu)$, the set of
isomorphism classes of regular supercuspidal data mapping to it is a
torsor under $H^1(u\ra W,Z\ra S)$. See {\em loc. cit.} Definition 5.3.3 and the paragraph below it.
\end{remark}

According to \cite[\S5.3]{kal}, for a regular supercuspidal datum
$(S,\wh{j},\chi,\mu,(G',\xi,z),j)$, in order to define the associated $L$-packet, we have to first modify it to be a proper datum
$(S,\wh{j},\chi^\new,\mu^\new,(G',\xi,z),j)$ which is isomorphic to $(S,\wh{j},\chi,\mu,(G',\xi,z),j)$. 
We refer to {\em loc. cit.}, especially pages 1076 and 1153, for the reason of this modification, and refer to {\em loc. cit.}  Steps 1 and 2 in \S5.3 for the definitions of $\chi^\new$ and $\mu^\new$. The following two definitions are given at {\em loc. cit.} page 1154.

\begin{defn}\label{defn. representations in packets}
Given a regular supercuspidal datum
$(S,\wh{j},\chi,\mu,(G',\xi,z),j)$, we set $\pi_{(S_j,\mu_j)}$ to be
the regular supercuspidal representation of $G'(F)$ associated to
the tame regular elliptic pair $(S_j,\mu_j)$ where
\begin{itemize}
\item $S_j$ is the image of $S$ in $G'$ under $j$,
\item $\mu_j:=(\mu^\new\circ j^{-1})\cdot\epsilon_{f,\ram}\cdot\epsilon^\ram$
 is a character of $S_j(F)$. Here $\epsilon_{f,\ram}$ and $\epsilon^\ram$
are certain quadratic characters of $S_j(F)$, whose definitions are given in
\cite[Definition 4.7.3 and (4.3.3)]{kal} respectively.
\end{itemize}
To avoid confusion, we will also write $\epsilon_{f,\ram,S_j}$ and $\epsilon^\ram_{S_j}$ instead of $\epsilon_{f,\ram}$ and $\epsilon^\ram$ respectively to indicate that they are characters of $S_j(F)$.
\end{defn}

\begin{defn}\label{defn. L-packets}
Let $\varphi$ be a regular supercuspidal $L$-parameter and
$(S,\wh{j},\chi,\mu)$ a regular supercuspidal $L$-datum
corresponding to $\varphi$. For each rigid inner twist $(G',\xi,z)$,
we define the $L$-packet $\Pi_\varphi(G')$ to be
$$\Pi_\varphi(G')=\{\pi_j\}$$ where
$(S,\wh{j},\chi,\mu,(G',\xi,z),j)$ runs over the set of isomorphism
classes of regular supercuspidal data mapping to
$(S,\wh{j},\chi,\mu)$ and $\pi_j:=\pi_{(S_j,\mu_j)}$. We define the
compound $L$-packet $\Pi_\varphi$ to be the disjoint union
$$\Pi_\varphi=\bigsqcup\Pi_\varphi(G')$$ where $(G',\xi,z)$ runs
over the set of isomorphism classes of rigid inner twists of $G$.
\end{defn}

\subsection{Twisted regular supercuspidal $L$-packets}\label{subsec. twisted
L-packet}

Now let $\theta$ be
an involution of $G$ and $H=G^\theta$. We denote by $\wh{\theta}$
the involution of $\wh{G}$ dual to $\theta$ with respect to the
fixed splittings. Note that $\wh{\theta}$ commutes with the action
of $\Gamma$ on $\wh{G}$ and can be extended to an $L$-automorphism
${^L\theta}:=\wh{\theta}\times\id_{W_F}$ of ${^LG}$. We refer to \cite[\S1.8]{kot84} for the existence and the basic properties of $\wh{\theta}$. We fix a
regular supercuspidal $L$-parameter $\varphi$ for $G$.

\subsubsection{Twisted regular supercuspidal $L$-parameters and $L$-packet data}
\label{subsubsec. twisted regular sc parameter and datum}

Suppose that $S$ is a maximal torus of $G$, and
$\chi=(\chi_\alpha)_{\alpha\in R(S,G)}$ is $\chi$-data for $R(S,G)$.
For $\alpha\in R(S,G)$, set $\theta(\alpha)=\alpha\circ\theta$,
which is an algebraic character of $\theta(S)$. Then $\theta(\alpha)$
is in $R(\theta(S),G)$, whose root space is $\theta(\fg_\alpha)$.
Hence we obtain a 1-1 correspondence
$$R(S,G)\longleftrightarrow R(\theta(S),G),\quad
\alpha\leftrightarrow\theta(\alpha).$$ Since $\theta$ is defined
over $F$, we have $\Gamma_\alpha=\Gamma_{\theta(\alpha)}$, and thus
$F_\alpha=F_{\theta(\alpha)}$ and
$F_{\pm\alpha}=F_{\pm{\theta(\alpha)}}$ for any $\alpha\in R(S,G)$.
Therefore
$\theta(\chi):=\left(\chi_{\theta(\alpha)}\right)_{\alpha\in
R(S,G)}$ is $\chi$-data for $R(\theta(S),G)$, where
$\chi_{\theta(\alpha)}:F^\times_{\theta(\alpha)}\ra\BC^\times$ is
the character $\chi_\alpha$ by identifying $F_{\theta(\alpha)}=F_\alpha$.

\begin{lem}\label{lem. involution twisted parameter}
The $L$-parameter ${^L\theta}\circ\varphi$ is regular supercuspidal.
Moreover, if $(S,\wh{j},\chi,\mu)$ is a regular supercuspidal
$L$-packet datum corresponding to $\varphi$, then
$(S,\wh{\theta}\circ\wh{j},\theta(\chi),\mu)$ is a regular
supercuspidal $L$-packet datum corresponding to
${^L\theta}\circ\varphi$.
\end{lem}

\begin{proof}
The first assertion, that ${^L\theta}\circ\varphi$ is a regular
supercuspidal $L$-parameter, can be easily verified by checking the
definition.

For the second assertion, it is harmless to assume that the fixed
splitting of $\wh{G}$ satisfies $\wh{j}(\wh{S})=\wh{T}$. Let $j:S\ra
G$ be an admissible embedding over $F$ with respect to $\wh{j}$.
Then $\theta\circ j:S\ra G$ is an admissible embedding over $F$ with
respect to $\wh{\theta}\circ\wh{j}:\wh{S}\ra\wh{G}$. We view $S$ as
a tame regular elliptic maximal torus of $G$ by the embedding $j$.
Then $(\theta(S),\mu\circ\theta)$ is a tame extra regular elliptic
pair for $G$, and $\theta(\chi)$ is
$\Omega(\theta(S),\theta(G^0))$-invariant minimally ramified
$\chi$-data for $R(\theta(S),G)$. In summary,
$(S,\wh{\theta}\circ\wh{j},\theta(\chi),\mu)$ is a regular
supercuspidal $L$-packet datum. It is routine to check that
$${^L\theta}\circ{^Lj_\chi}={^Lj}_{\theta(\chi)}$$ where
${^Lj_{\theta(\chi)}}$ is the $L$-embedding ${^LS}\ra{^LG}$
extending $\wh{\theta}\circ\wh{j}$ that is determined by the
$\chi$-data $\theta(\chi)$. Therefore
$${^L\theta}\circ\varphi={^Lj_{\theta(\chi)}}\circ\varphi_{S,\mu},$$
and thus $(S,\wh{\theta}\circ\wh{j},\theta(\chi),\mu)$ corresponds
to ${^L\theta}\circ\varphi$.
\end{proof}

\subsubsection{Rigid inner twists of symmetric spaces}
\label{subsubsec. rigid inner twists of sym space}
\begin{defn}\label{defn. rigid inner twists of sym space}
\begin{enumerate}
\item Let $(G',\xi,z)$ be a rigid inner twist of $G$. We call
$(G',\xi,z)$ a {\em rigid inner twist of $(G,H,\theta)$} if $z$ lies
in the image of $Z^1(u\ra W,Z\ra H)$ in $Z^1(u\ra W,Z\ra G)$.
\item Let $(G',\xi,z)$ be a rigid inner twist of $(G,H,\theta)$.
We define an
involution $\theta'$ of $G'$ by
$$\theta'=\xi\circ\theta\circ\xi^{-1}.$$
\end{enumerate}
\end{defn}

\begin{lem}\label{lem. involution on inner forms}
The involution $\theta'$ is defined over $F$.
\end{lem}
\begin{proof}
Let $\bar{z}$ be the image of $z$ in $Z^1(\Gamma,G/Z)$, which is
viewed as an element in $Z^1(\Gamma, H/Z)$ by the condition imposed
on $z$. Then for any $\sigma\in\Gamma$ we have
$$\begin{aligned}\sigma\circ\theta'&=\sigma\circ\xi\circ\theta\circ\xi^{-1}\\
&=\xi\circ\Int(\bar{z}_\sigma)\circ\sigma\circ\theta\circ\xi^{-1}\\
&=\xi\circ\Int(\theta(\bar{z}_\sigma))\circ\theta\circ\sigma\circ\xi^{-1}\\
&=\xi\circ\theta\circ\Int(\bar{z}_\sigma)\circ\sigma\circ\xi^{-1}\\
&=\xi\circ\theta\circ\xi^{-1}\circ\sigma\\
&=\theta'\circ\sigma.
\end{aligned}$$
Therefore $\theta'$ is defined over $F$.
\end{proof}

\begin{remark}\label{rem. inner twists of involutions}
Note that,
according to the definition of $\theta'$, we have
$$\theta'\circ\xi=\xi\circ\theta.$$
Let $H'=(G')^{\theta'}$. Then we have
$$H'(\bar{F})=G'(\bar{F})^{\theta'}=\left(\xi\left(G(\bar{F})\right)\right)^{\theta'}=
\xi\left(G(\bar{F})^{\theta}\right)=\xi\left(H(\bar{F})\right).$$ Thus
the restriction of $(\xi,z)$ onto $H$ gives rise to a rigid inner
twist $\xi_H:H\ra H'$. If $(G',\xi,z)$ is clear, we also call
$(G',H',\theta')$ a rigid inner twist of $(G,H,\theta)$.
\end{remark}

\subsubsection{Twisted regular supercuspidal $L$-packets}

\begin{defn}\label{defn. twisted L-packets}
Let $\varphi$ be a regular supercuspidal $L$-parameter and
$(S,\wh{j},\chi,\mu)$ a regular supercuspidal $L$-datum
corresponding to $\varphi$. For each rigid inner twist
$(G',H',\theta')$ of $(G,H,\theta)$, we define the {\em twisted
$L$-packet} $\Pi^\theta_\varphi(G')$ to be
$$\Pi^\theta_\varphi(G')=\{\pi_j\circ\theta'| \pi_j\in\Pi_\varphi(G')\},$$  and define the
{\em compound twisted $L$-packet} $\Pi_\varphi^{\theta,\circ}$ to be the disjoint union
$$\Pi^{\theta,\circ}_\varphi=\bigsqcup\Pi^{\theta,\circ}_\varphi(G')$$ where $(G',H',\theta')$ runs
over the set of isomorphism classes of rigid inner twists of
$(G,H,\theta)$.
\end{defn}

The way we define the twisted $L$-packet
$\Pi_\varphi^{\theta,\circ}$ is on the level of representations,
that is, we twist the representations in the $L$-packets by
involutions. It is natural to ask whether the twisted $L$-packet
$\Pi_\varphi^{\theta,\circ}$ is indeed a compound $L$-packet in some
sense. The answer is yes. More precisely we have:

\begin{prop}\label{prop. twisted L-packets}
For each rigid inner twist $(G',H',\theta')$ we have
$$\Pi_\varphi^\theta(G')=\Pi_{{^L\theta\circ\varphi}}(G').$$
Therefore we have
$\Pi_\varphi^{\theta,\circ}\subseteq\Pi_{{^L\theta\circ\varphi}}$.
\end{prop}

\begin{proof}
Let $(S,\wh{j},\chi,\mu)$ be a regular supercuspidal $L$-datum
corresponding to $\varphi$ and $(S,\wh{j},\chi,\mu,(G',\xi,z),j)$ a
regular supercuspidal datum  such that $(G',\xi,z)$ is a rigid inner
twist of $(G,H,\theta)$. According to Lemma \ref{lem. involution
twisted parameter}, $(S,\wh{\theta}\circ\wh{j},\theta(\chi),\mu)$ is
a regular supercuspidal $L$-datum corresponding to
$^L\theta\circ\varphi$. Choose a $\Gamma$-fixed admissible embedding
$j_0:S\ra G$ with respect to $\wh{j}$. Then $\theta\circ j_0:S\ra G$
is a $\Gamma$-fixed admissible embedding with respect to
$\wh{\theta}\circ\wh{j}$. Since $j:S\ra G'$ is admissible, there
exists $g\in G$ such that $j=\xi\circ\Int(g)\circ j_0$. We have
$$\begin{aligned}
\theta'\circ j&=\theta'\circ\xi\circ\Int(g)\circ j_0\\
&=\xi\circ\theta\circ\Int(g)\circ j_0\\
&=\xi\circ\Int(\theta(g))\circ(\theta\circ j_0).
\end{aligned}$$
Hence $\theta'\circ j:S\ra G'$ is indeed an admissible embedding
with respect to $\wh{\theta}\circ\wh{j}$. Therefore, for those rigid
inner twists $(G',\xi,z)$ of $(G,H,\theta)$, the map
$$(S,\wh{j},\chi,\mu,(G',\xi,z),j)\mapsto (S,\wh{\theta}\circ\wh{j},
\theta(\chi),\mu,(G',\xi,z),\theta'\circ j)$$ establishes an 1-1
correspondence between regular supercuspidal data for $\varphi$ with
regular supercuspidal data for ${^L\theta}\circ\varphi$. We remark that if $(S,\wh{j},\chi^\new,\mu^\new,(G',\xi,z),j)$ is the modified datum of $(S,\wh{j},\chi,\mu,(G',\xi,z),j)$, then $(S,\wh{\theta}\circ\wh{j},
\theta(\chi^\new),\mu^\new,(G',\xi,z),\theta'\circ j)$ equals to the modified datum of
$(S,\wh{\theta}\circ\wh{j},
\theta(\chi),\mu,(G',\xi,z),\theta'\circ j)$. To prove the
proposition, it remains to show that
$$\pi_{(S_j,\mu_j)}\circ\theta'\simeq\pi_{(S_{\theta'\circ j},\mu_{\theta'\circ j})}.$$
First we have
$$\pi_{(S_j,\mu_j)}\circ\theta'\simeq\pi_{(\theta'(S_j),\mu_j\circ\theta')}.$$
As a character of $\theta'(S_j)(F)$,
$$\mu_j\circ\theta'=(\mu^\new\circ j^{-1}\circ\theta')\cdot(\epsilon_{f,\ram,S_j}\circ\theta')
\cdot(\epsilon^\ram_{S_j}\circ\theta').$$ On the other hand, we have
$$\mu_{\theta'\circ j}=(\mu^\new\circ
j^{-1}\circ\theta')\cdot\epsilon_{f,\ram,\theta'(S_j)}\cdot
\epsilon^\ram_{\theta'(S_j)}.$$ According to the definition of
$\epsilon_{f,\ram}$ and $\epsilon^\ram$, and the correspondence
$R(S,G')\leftrightarrow R(\theta'(S),G')$ established before, it is
straightforward to check that
$$\epsilon_{f,\ram,S_j}\circ\theta'=\epsilon_{f,\ram,\theta'(S_j)}\quad
\textrm{and}\quad
\epsilon^\ram_{S_j}\circ\theta'=\epsilon^\ram_{\theta'(S_j)},$$ which
completes the proof.

\end{proof}

\subsection{Contragredient regular supercuspidal $L$-packets}\label{subsec.
contragredient}\label{subsec. contragredient L-packets}

For a general $L$-parameter $\varphi$ for $G$, it is conjectured by
Adams and Vogan \cite{av16} on the level of packets, and by Prasad
\cite{pra} and Kaletha \cite{ka13} on the level of representations,
that the contragredient of the $L$-packet $\Pi_\varphi$ itself
should be an $L$-packet and there should exist an explicit relation
between the refined $L$-parameters of $\pi$ and $\pi^\vee$ for
$\pi\in\Pi_\varphi$. Kaletha \cite[\S5]{ka13} showed that this conjecture
holds when $\varphi$ is {\em tame regular semisimple elliptic} or
{\em epipelagic}. In this subsection, we give a proof of this
conjecture for regular supercuspidal parameters, following the
arguments of \cite[\S5]{ka13} closely. From now on, we fix a regular
supercuspidal $L$-parameter $\varphi$ for $G$.

\begin{defn}\label{defn. contragredient}
For each rigid inner twist $(G',\xi,z)$ of $G$, the {\em
contragredient $L$-packet} $\Pi^\vee_\varphi(G')$ is defined to be
$$\Pi^\vee_\varphi(G')=\left\{\pi_j^\vee|
\pi_j\in\Pi_\varphi(G')\right\},$$ and the {\em compound
contragredient $L$-packet} $\Pi_\varphi^\vee$ is defined to be the disjoint union
$$\Pi_\varphi^\vee=\bigsqcup\Pi_\varphi^\vee(G')$$ where $(G',\xi,z)$ runs
over the set of isomorphism classes of rigid inner twists of $G$.
\end{defn}

We fix a $\Gamma$-invariant splitting
$(\wh{T},\wh{B},\{X_{\wh{\alpha}}\})$ for $\wh{G}$. The {\em
Chevalley involution $\wh{C}$} of $\wh{G}$ is uniquely determined by
the following conditions:
\begin{itemize}
\item $\wh{C}(\wh{T})=\wh{T}$ and $\wh{C}|_{\wh{T}}=-1$ where -1
denotes the inverse map,
\item $\wh{C}(\wh{B})=\wh{B}^\op$ where $\wh{B}^\op$ is the opposite
Borel of $\wh{B}$,
\item $C(X_{\wh{\alpha}})=X_{-\wh{\alpha}}$ for $\wh{\alpha}\in
R(\wh{T},\wh{G})$.
\end{itemize}
Note that $\wh{C}$ commutes with the action of $\Gamma$ on $\wh{G}$.
Thus we can extend $\wh{C}$ to an $L$-automorphism
${^LC}=\wh{C}\times\id_{W_F}$ of ${^LG}$.

Let $(S,\wh{j},\chi,\mu)$ be a regular supercuspidal $L$-packet
datum corresponding to $\varphi$. We assume that
$\wh{j}(\wh{S})=\wh{T}$. Note that ${^LC}\circ\varphi$ is also regular
supercuspidal. For the $\chi$-data $\chi=(\chi_\alpha)$ we denote by
$\chi^{-1}$ the $\chi$-data $(\chi^{-1}_\alpha)$.

\begin{lem}\label{lem. contragredient L-packet datum}
The 4-tuple $(S,\wh{j},\chi^{-1},\mu^{-1})$ is a regular
supercuspidal $L$-packet datum and corresponds to
${^LC\circ\varphi}$.
\end{lem}

\begin{proof}
It is straightforward to check that $(S,\wh{j},\chi^{-1},\mu^{-1})$
is a regular supercuspidal $L$-packet datum. On the other hand, by
\cite[Lemma 4.1]{ka13} the following diagram is commutative:
$$\begin{tikzcd}
{^LS} \arrow[rr, "-1"] \arrow[dd, "{^Lj_\chi}"'] &  & {^LS} \arrow[dd, "{^Lj_{\chi^{-1}}}"] \\
&  &                  \\
{^LG} \arrow[rr, "{^LC}"]            &  & {^LG}           
\end{tikzcd} $$
Therefore, since $\varphi={^Lj_\chi}\circ\varphi_{S,\mu}$, we have
$$\begin{aligned}{^LC\circ\varphi}&={^Lj_{\chi^{-1}}}\circ(-1)\circ\varphi_{S,\mu}\\
&={^Lj_{\chi^{-1}}}\circ\varphi_{S,\mu^{-1}},
\end{aligned}$$
where $\varphi_{S,\mu^{-1}}:W_F\ra{^LS}$ is the $L$-parameter
attached to the character $\mu^{-1}$ of $S(F)$. This implies that
$(S,\wh{j},\chi^{-1},\mu^{-1})$ corresponds to ${^LC}\circ\varphi$.
\end{proof}

\begin{prop}\label{prop. contragredient L-packet}
We have $\Pi_\varphi^\vee=\Pi_{{^LC\circ\varphi}}$.
\end{prop}

\begin{proof}
Let $(S,\wh{j},\chi,\mu,(G',\xi,z),j)$ be a regular supercuspidal
datum of $\varphi$.  As before, we remark that if $(S,\wh{j},\chi^\new,\mu^\new,(G',\xi,z),j)$ is the modified datum of
$(S,\wh{j},\chi,\mu,(G',\xi,z),j)$, then $(S,\wh{j},(\chi^\new)^{-1},(\mu^\new)^{-1},(G',\xi,z),j)$ equals to the modified datum of
$(S,\wh{j},\chi^{-1},\mu^{-1},(G',\xi,z),j)$.  According to \cite[\S3]{hm08}, we have
$$\pi_{(S_j,\mu_j)}^\vee\simeq\pi_{(S_j,(\mu_j)^{-1})}.$$ Recall that $\mu_j=(\mu^\new\circ j^{-1})\cdot
\epsilon_{f,\ram}\cdot\epsilon^\ram$, and both
$\epsilon_{f,\ram}$ and $\epsilon^\ram$ are quadratic characters
of $S_j(F)$. Hence we have
$$(\mu_j)^{-1}=\left((\mu^\new)^{-1}\circ j^{-1}\right)\cdot
\epsilon_{f,\ram}\cdot\epsilon^\ram=(\mu^{-1})_j.$$ Therefore,
the representation $\pi_{(S_j,(\mu^{-1})_j)}$, which is attached to
the regular supercuspidal datum
$(S,\wh{j},\chi^{-1},\mu^{-1},(G',\xi,z),j)$ for
${^LC\circ\varphi}$, is isomorphic to $\pi_{(S_j,\mu_j)}^\vee$.
\end{proof}

\subsection{Consequences}\label{subsec. consequences}
Let $\theta$ be an involution of $G$. Let $\varphi$ be a regular
depth-zero or an epipelagic supercuspidal $L$-parameter for $G$,
which is in particular a regular supercuspidal parameter. Regular
depth-zero supercuspidal $L$-parameters were first introduced in
\cite[page 825]{dr09}, which are called tame regular semisimple
elliptic $L$-parameters therein. Epipelagic supercuspidal
$L$-parameters were first considered in \cite[\S7]{ry14} and
then discussed in \cite[\S5.1]{kal15}. Recall the fact that
regular depth-zero supercuspidal $L$-parameters
correspond to regular depth-zero supercuspidal representations, and
epipelagic supercuspidal $L$-parameters correspond to epipelagic
supercuspidal representations. The following corollary is a direct
consequence of Corollaries \ref{cor. cor of main theorem} and
\ref{cor. epipelagic}, Propositions \ref{prop. twisted L-packets}
and \ref{prop. contragredient L-packet}.

\begin{cor}\label{cor. consequence}
Let $(G',H',\theta')$ be a rigid inner twist of $(G,H,\theta)$, $\varphi$
a regular depth-zero or an epipelagic supercuspidal $L$-parameter, and
$\pi\in\Pi_\varphi(G')$. Suppose that $\pi$ is $H'$-distinguished.
Then the $L$-parameters $^L\theta\circ\varphi$ and $^LC\circ\varphi$ are
$\wh{G}$-conjugate, and thus
$\Pi_{^L\theta\circ\varphi}=\Pi_{^LC\circ\varphi}$.

\end{cor}

\s{\small Chong Zhang\\
Department of Mathematics, Nanjing University,\\
Nanjing 210093, Jiangsu, P. R. China.\\
E-mail address: \texttt{zhangchong@nju.edu.cn}}

\end{document}